\documentclass[12pt, reqno]{amsart}
\usepackage{amsmath, amsthm, amscd, amsfonts, amssymb, graphicx, color, mathrsfs}
\usepackage[bookmarksnumbered, colorlinks, plainpages]{hyperref}
\usepackage[all]{xy} 
\usepackage{slashed}

\newtheorem{theorem}{Theorem}[section]
\newtheorem{lemma}[theorem]{Lemma}
\newtheorem{proposition}[theorem]{Proposition}

\theoremstyle{definition}
\newtheorem{definition}[theorem]{Definition}
\newtheorem{example}[theorem]{Example}

\theoremstyle{remark}
\newtheorem{remark}[theorem]{Remark}
\numberwithin{equation}{section}

\begin{document}
\setcounter{page}{1}

\title[Nikolskii inequality, Besov spaces and multipliers]{Littlewood-Paley theorem, Nikolskii inequality, Besov spaces, Fourier and spectral  multipliers on graded Lie groups}

\author[D. Cardona]{Duv\'an Cardona}
\address{
  Duv\'an Cardona:
  \endgraf
  Department of Mathematics: Analysis, Logic and Discrete Mathematics
  \endgraf
  Ghent University, Belgium
  \endgraf
  {\it E-mail address} {\rm duvanc306@gmail.com,\, Duvan.CardonaSanchez@ugent.be }
  }

\author[M. Ruzhansky]{Michael Ruzhansky}
\address{
  Michael Ruzhansky:
  \endgraf
  Department of Mathematics: Analysis, Logic and Discrete Mathematics
  \endgraf
  Ghent University, Belgium
  \endgraf
 and
  \endgraf
  School of Mathematical Sciences
  \endgraf
  Queen Mary University of London
  \endgraf
  United Kingdom
  \endgraf
  {\it E-mail address} {\rm Michael.Ruzhansky@ugent.be}
  }

\thanks{The second author was supported in parts by the EPSRC
 grant EP/K039407/1 and by the Leverhulme Grant RPG-2014-02. 
 The authors were supported by the FWO Odysseus 1 grant G.0H94.18N: Analysis and Partial Differential Equations.
}


\begin{abstract}
In this paper we investigate Besov spaces on graded Lie groups. We prove a Nikolskii type inequality (or the Reverse H\"older inequality) on graded Lie groups and as consequence we obtain embeddings of Besov spaces. We prove a version of the Littlewood-Paley theorem on  graded Lie groups. The results are applied to obtain embedding properties of Besov spaces and multiplier theorems for both spectral and Fourier multipliers in Besov spaces on graded Lie groups. In particular, we  give a number of sufficient conditions for the boundedness of Fourier multipliers in Besov spaces.
\end{abstract} 

\maketitle

\tableofcontents
\allowdisplaybreaks

\section{Introduction} 

In this paper we are interested in advancing the notions and results of harmonic analysis in the setting of graded Lie groups, building up on the fundamental book \cite{FE} of Folland and Stein, as well as on more recent developments over the decades, in particular summarised in the recent book \cite{FR2} by V\'eronique Fischer and the second author. Indeed, as it was pointed out by Folland and Stein, the setting of homogeneous groups is ideal for the distillation of those results of harmonic analysis that depend only on the group and dilation structures of the underlying space, while the setting of graded Lie groups allows one to also use the techniques coming from the theory of partial differential operators. The difference between the classes of nilpotent, homogeneous and graded Lie groups is rather small, with the majority of nilpotent Lie groups allowing for a compatible graded structure, see \cite[Chapter 3]{FR2} for a detailed explanation. 
In particular, this setting includes the class of stratified groups (\cite{folland_75}) when the Rockland operator can be chosen to be the sub-Laplacian. We also mention that general Rockland operators on graded Lie groups naturally appear when one is dealing with questions concerning general partial differential operators on manifolds, as their liftings following the celebrated lifting procedure of Rothschild and Stein \cite{RS}.

\smallskip
Summarising the research of this paper, here we obtain the following results:

\begin{itemize}
\item establish the Nikolskii (or the reverse H\"older) inequality in the setting of graded Lie groups in terms of its homogeneous dimension. We believe such a result to be new already on stratified groups, and even on the Heisenberg group;
\item prove the Littlewood-Paley theorem on graded Lie groups for the dyadic decomposition associated to positive Rockland operators;
\item investigate homogeneous and inhomogeneous Besov spaces in terms of Rockland operators and prove their embedding properties. We show that the Besov spaces in this context are also the interpolation spaces between Sobolev spaces, and prove that they are independent of a particular choice of the Rockland operator used to define them. We also prove their embedding properties with the usual (locally defined) Besov spaces on $\mathbb R^n$;
\item apply these results to establish multiplier theorems for spectral and Fourier multipliers in Besov spaces on graded Lie groups. More precisely, we give negative results on the boundedness of invariant operators in Besov spaces. For Fourier multipliers, we show that the boundedness between $L^{p}$-spaces implies the boundedness on Besov spaces and give several applications of this result to Fourier multipliers using H\"ormander-Mihlin type and other theorems in this setting.
\end{itemize}

Nikolskii-type inequalities, following the usual terminology, are, roughly speaking,
inequalities between different metrics of the same function (usually trigonometric polynomials). Nikolskii \cite{Nik1} in 1951 proved the inequalities for $1\leq p\leq q\leq \infty$:
\begin{equation}
\Vert T_{L_1,L_2,...,L_n}\Vert_{L^{q}[0,2\pi]}\leq 2^{n}{[(2\pi)^{n}L_1 L_2 \cdots  L_n ] }^{\frac{1}{p}-\frac{1}{q}}\Vert T_{L_1,L_2,...,L_n}\Vert_{L^{p}[0,2\pi]},
\end{equation}
for trigonometric polynomials of the form
\begin{equation}
T_{L_1,L_2,...,L_n}=\sum_{k=1}^{n}\sum_{j_{k}=-L_k}^{L_k}c_{j_1,j_2,...,j_k}e^{i(j_1x_1+\cdots j_kx_{k})},
\end{equation}
as well as for entire functions of exponential type. 
Sometimes such inequality is also called the reverse H\"older inequality in the literature.

On $\mathbb{R}^n,$ the Nikolskii inequality takes the form
\begin{equation}
\Vert f\Vert_{L^{q}(\mathbb{R}^n)}\leq C[\mu[\textrm{c.h.}[\text{supp}(\widehat{f})]]]^{\frac{1}{p}-\frac{1}{q}}\Vert f\Vert_{L^{p}(\mathbb{R}^n)},
\end{equation}
for every function $f\in{L}^p(\mathbb{R}^n)$ with Fourier transform $\widehat{f}$ of compact support. Here,  $\textrm{c.h.(E)}$ denotes the convex hull of the set $E.$ Recently, the Nikolskii inequality has been considered in the setting of Lie groups $G.$ In \cite{Pe2}, Pesenson has obtained the Nikolskii inequality  for symmetric spaces $G/K$ of non-compact type. On the other hand, for compact homogeneous manifolds $G/K,$ in \cite{NRT} the following Nikolskii inequality was obtained:
\begin{equation}
\Vert T_{L} \Vert_{L^{q}(G/K)}\leq N(\rho L)^{\frac{1}{p}-\frac{1}{q}}\Vert T_{L} \Vert_{L^{p}(G/K)},
\end{equation}
for $0<p<q\leq \infty;$ here, if $0<p\leq 2,$ $\rho:=1$, and for $2<p\leq \infty,$ $\rho:=[\frac{p}{2}]+1,$  $N(L)\simeq L^{\dim G/K}$ is the Weyl eigenvalue counting function for the elliptic pseudo-differential operator $(I-\mathcal{L}_{G/K})^{\frac{1}{2}}$, where $\mathcal{L}_{G/K}$ is the Laplacian on ${G/K}.$ 

In this paper we prove a Nikolskii type inequality in the framework of graded Lie groups $G$. 
We believe this to be new also on stratified groups, even on the Heisenberg group.

This inequality is important in mathematical analysis because it is a fundamental tool in the proof of several embeddings properties of important function spaces such as Besov spaces. The Besov spaces form scales $B^r_{p,q}(G)$ carrying three indices $r\in\mathbb{R},$ $0<p,q\leq \infty,$ and they can be obtained by interpolation of suitable Sobolev spaces. As it was discussed in \cite{FR3}, Sobolev spaces can be defined on $\mathbb{R}^n,$ and on compact and non-compact Lie groups in various equivalent ways. In a recent work of the second author with V. Fischer, Sobolev spaces were introduced on arbitrary graded Lie groups by using positive Rockland operators (see \cite{FR3}). It is important to mention that Sobolev spaces on stratified Lie groups were introduced by Folland in \cite {folland_75} by using sub-Laplacians, and it was proved (see also \cite{FE}) that these spaces are different from their Euclidean counterpart defined by the Fourier transform or by using the local properties of the Laplace operators. The Folland's Sobolev spaces coincide with those introduced in \cite{FR3} on graded Lie groups in the setting of stratified groups. 
We also refer to \cite{BFG_refined} for a number of useful inequalities on graded Lie groups.

In this paper we use positive Rockland  operators in order to introduce Besov spaces on graded Lie groups, and later on, we prove that our Besov spaces can be obtained by interpolation of the Sobolev spaces introduced in \cite{FR3}. For special cases of parameters $p,q$ and $r,$ Besov spaces were also considered by Bahouri, G\'erard and Xu in \cite{bahouri}.  Apart of the trivial embeddings that can be obtained on the $q$ parameters for Besov spaces $B^{r}_{p,q}(G),$ the Nikolskii inequality will be a useful tool in order to establish embeddings that involve  the parameters $r$ and $p.$ 

As a substitute of the Plancherel theorem on $L^2(G),$ in $L^{p}(G)$ spaces, we prove a version of the Littlewood-Paley theorem and we will use both, our Nikolskii inequality and our Littlewood-Paley theorem in order to get boundedness of Fourier multipliers and spectral multipliers on Besov spaces.  For the case of Fourier multipliers we will use the version of the 
H\"ormander-Mihlin theorem in the nilpotent setting \cite{FR}.  

We note that in the case of the sub-Laplacian, a wealth of results is available, to mention only a few, see e.g. 
Folland \cite{folland_75} and Saka \cite{saka} for Sobolev spaces and Besov spaces on
stratified groups, respectively; Furioli, Melzi and Veneruso \cite{furioli} and Alexopoulos \cite{alexo} for the Littlewood-Paley theorem and Besov spaces, and for spectral multiplier theorems for the sub-Laplacian on Lie groups of polynomial growth, respectively. 
There are also many results on functions of sub-Laplacians in the fundamental monograph
by Varopoulos, Saloff-Coste and Coulhon \cite{VSC}.

The novelty of this paper is that we are working with Rockland operators; these are linear invariant homogeneous hypoelliptic partial differential operators, in view of the Helffer and Nourrigat's resolution of the Rockland conjecture in \cite{helffer+nourrigat-79}. Such operators always exist on graded Lie groups and, in fact, the existence of such operators on nilpotent Lie groups does characterise the class of graded Lie groups, see \cite[Section 4.1]{FR2} for further details and references. As the literature concerning the analysis based on sub-Laplacians is immense, we do not review it here, but refer to the introduction in \cite{FR2} for a more extensive presentation of the subject.
Some results of this paper were announced in \cite{CR}.

This paper is organised as follows. In Section \ref{Preliminaries} we present some preliminaries on the Fourier analysis of graded Lie groups and its homogeneous structure, and we present positive Rockland operators and elements of their functional calculus. For this we follow \cite{FR2}. In Section \ref{Nikolskii} we prove our version of the Nikolskii inequality for functions defined on graded Lie groups. In Section \ref{Littlewood-Paley} we prove our version of the Littlewood-Paley theorem. In Section \ref{Besov} we define Besov spaces and we prove some embedding properties for these spaces. In Section \ref{independenceinterpolation} we prove that Besov spaces can be obtained by interpolation of Sobolev spaces in the nilpotent setting and in Subsection \ref{sec:int} we prove further interpolation properties.
In Section \ref{SEC:local} we show embedding properties between localisation of these Besov spaces and the usual (Euclidean) Besov spaces.

Finally, in Section \ref{Multipliers} we study the boundedness of Fourier multipliers and spectral multipliers in Besov spaces. %For the case of spectral multipliers we generalise a classical result by Marcinkiewicz, which asserts that if $m\in \textnormal{BV}(\mathbb{R})$ is a bounded function which has uniformly bounded variation on every dyadic interval of $\mathbb{R}$, then $m$ is a multiplier on $L^p(\mathbb{R})$ for every $1<p<\infty.$ We prove that our version of the Marcinkiewicz theorem generates spectral multipliers bounded on Besov spaces.%
In the case of Fourier multipliers, we prove that $L^p(G)$-multipliers on graded nilpotent Lie groups generate multipliers in Besov spaces $B^{r}_{p,q}(G)$. As a consequence of this fact, we end Section \ref{Multipliers} with several examples on multipliers.

The authors would like to thank Alessio Martini for valuable remarks on the original version of the  manuscript.

\section{Preliminaries}\label{Preliminaries}

\noindent In this section, we recall  some preliminaries  on graded and homogeneous Lie groups $G$. The unitary dual of these groups will be denoted by $\widehat{G}$. We  also present the notion of Rockland operators and Sobolev spaces on $G$ and on the unitary dual $\widehat{G}$ by following \cite{FR}, to which we refer for further details on constructions presented in this section.

\subsection{Homogeneous and graded Lie groups} Let $G$ be a graded Lie group. This means that $G$ is a connected and simply connected nilpotent Lie group whose Lie algebra $\mathfrak{g}$ may be decomposed as the sum of subspaces $\mathfrak{g}=\mathfrak{g}_{1}\oplus\mathfrak{g}_{2}\oplus \cdots \oplus \mathfrak{g}_{s}$ such that $[\mathfrak{g}_{i},\mathfrak{g}_{j} ]\subset \mathfrak{g}_{i+j},$ and $ \mathfrak{g}_{i+j}=\{0\}$ if $i+j>s.$ This implies that the group $G$ is nilpotent because the sequence
$$ \mathfrak{g}_{(1)}:=\mathfrak{g},\,\,\,\,\mathfrak{g}_{(n)}:=[\mathfrak{g},\mathfrak{g}_{(n-1)}] $$
defined inductively terminates at $\{0\}$ in a finite number of steps. Examples of such groups are the Heisenberg group $\mathbb{H}^n$ and more generally any stratified groups where the Lie algebra $ \mathfrak{g}$ is generated by $\mathfrak{g}_{1}$. The exponential mapping from $\mathfrak{g}$ to $G$ is  a diffeomorphism, then, we can identify $G$ with $\mathbb{R}^n$ or $\mathfrak{g}_{1}\times \mathfrak{g}_{2}\times \cdots \times \mathfrak{g}_{s}$ as  manifolds. Consequently we denote by $\mathscr{S}(G)$ the Schwartz space of functions on $G,$ by considering the identification $G\equiv \mathbb{R}^n.$ Here, $n$ is the topological dimension of $G,$ $n=n_{1}+\cdots +n_{s},$ where $n_{k}=\mbox{dim}\mathfrak{g}_{k}.$ A family of dilations $D_{r},$ $r>0,$ on a Lie algebra $\mathfrak{g}$ is a family of linear mappings from $\mathfrak{g}$ to itself satisfying the following two conditions:
\begin{itemize}
\item For every $r>0,$ $D_{r}$ is a map of the form
$$ D_{r}=\textnormal{Exp}(\ln(r)A) $$
for some diagonalisable linear operator $A$ on $\mathfrak{g}.$
\item $\forall X,Y\in \mathfrak{g}, $ and $r>0,$ $[D_{r}X, D_{r}Y]=D_{r}[X,Y].$ 
\end{itemize}
We call  the eigenvalues of $A,$ $\nu_1,\nu_2,\cdots,\nu_n,$ the dilations weights or weights of $G$. A homogeneous Lie group is a connected simply connected Lie group whose Lie algebra $\mathfrak{g}$ is equipped with a family of dilations $D_{r}.$ In  such case, and with the notation above,  the homogeneous dimension of $G$ is given by  $$ Q=\textnormal{Tr}(A)=\sum_{l=1}^{s}l\cdot\dim \mathfrak{g}_l.  $$
We can transport dilations $D_{r}$ of the Lie algebra $\mathfrak{g}$ to the group by considering the family of maps
$$ \exp_{G}\circ D_{r} \circ \exp_{G}^{-1},\,\, r>0, $$
where $\exp_{G}:\mathfrak{g}\rightarrow G$ is the usual exponential function associated to the Lie group $G.$ We denote  this family of dilations also by $D_{r}$ and we refer to them as dilations on the group. If we write $rx=D_{r}(x),$ $x\in G,$ $r>0,$ then a relation on the homogeneous structure of $G$ and the Haar measure $dx$ on $G$ is given by $$ \int_{G}(f\circ D_{r})(x)dx=r^{-Q}\int_{G}f(x)dx. $$

\subsection{The unitary dual and the Plancherel theorem} We will always equip a graded Lie group with the Haar measure $dx.$ For simplicity, we will write $L^p(G)$ for $L^p(G, dx).$ We denote by $\widehat{G}$  the unitary dual of $G,$ that is the set of equivalence classes of unitary, irreducible, strongly continuous representations of $G$ acting in separable Hilbert spaces. The unitary dual can be equipped with the Plancherel measure  $d\mu.$ So, the Fourier transform of every  function $\varphi\in \mathscr{S}(G)$ at $\pi\in \widehat{G}$ is defined by 
 $$  (\mathscr{F}_{G}\varphi)(\pi)\equiv\widehat{\varphi}(\pi)=\int_{G}\varphi(x)\pi(x)^*dx, $$
 and the corresponding Fourier inversion formula is given by
 $$  \varphi(x)=\int_{ \widehat{G}}\mathrm{Tr}(\pi(x)\widehat{\varphi}(\pi))d\mu(\pi).$$
In this case, we have the Plancherel identity
$$ \Vert \varphi \Vert_{L^2(G)}= \left(\int_{ \widehat{G}}\mathrm{Tr}(\widehat{\varphi}(\pi)\widehat{\varphi}(\pi)^*)d\mu(\pi) \right)^{\frac{1}{2}}=\Vert  \widehat{\varphi}\Vert_{ L^2(\widehat{G} ) } .$$
We also denote $\Vert\widehat{\varphi} \Vert^2_{\textnormal{HS}}=\textnormal{Tr}(\widehat{\varphi}(\pi)\widehat{\varphi}(\pi)^*)$ the Hilbert-Schmidt norm of operators.  A  Fourier multiplier  is formally defined by

\begin{equation}\label{mul}T_{\sigma}u(x)=\int_{ \widehat{G}}\mathrm{Tr}(\pi(x)\sigma(\pi)\widehat{f}(\pi))d\mu(\pi),
\end{equation}

\noindent where the symbol $\sigma(\pi)$  is defined on the unitary dual $\widehat{G}$ of $G.$ For a rather comprehensive treatment of this quantization we refer to \cite{FR2}  and to references therein. 
\subsection{Homogeneous linear operators and Rockland operators} A linear operator $T:\mathscr{D}(G)\rightarrow \mathscr{D}'(G)$ is homogeneous of  degree $\nu\in \mathbb{C}$ if for every $r>0$ 
\begin{equation}
T(f\circ D_{r})=r^{\nu}(Tf)\circ D_{r}
\end{equation}
holds for every $f\in \mathscr{D}(G). $
If for every representation $\pi\in\widehat{G},$ $\pi:G\rightarrow U(\mathcal{H}_{\pi}),$ we denote by $\mathcal{H}_{\pi}^{\infty}$ the set of smooth vectors, that is, the space of elements $v\in \mathcal{H}_{\pi}$ such that the function $x\mapsto \pi(x)v,$ $x\in \widehat{G}$ is smooth,  a Rockland operator is a left-invariant differential operator $\mathcal{R}$ which is homogeneous of positive degree $\nu=\nu_{\mathcal{R}}$ and such that, for every unitary irreducible non-trivial representation $\pi\in \widehat{G},$ $\pi(\mathcal{R})$ is injective on $\mathcal{H}_{\pi}^{\infty};$ $\sigma_{\mathcal{R}}(\pi)=\pi(\mathcal{R})$ is the symbol associated to $\mathcal{R}.$ It coincides with the infinitesimal representation of $\mathcal{R}$ as an element of the universal enveloping algebra. It can be shown that a Lie group $G$ is graded if and only if there exists a differential Rockland operator on $G.$ If the Rockland operator is formally self-adjoint, then $\mathcal{R}$ and $\pi(\mathcal{R})$ admit self-adjoint extensions on $L^{2}(G)$ and $\mathcal{H}_{\pi},$ respectively. Now if we preserve the same notation for their self-adjoint
extensions and we denote by $E$ and $E_{\pi}$  their spectral measures, by functional calculus we have
$$ \mathcal{R}=\int\limits_{-\infty}^{\infty}\lambda dE(\lambda),\,\,\,\textnormal{and}\,\,\,\pi(\mathcal{R})=\int\limits_{-\infty}^{\infty}\lambda dE_{\pi}(\lambda). $$
We now recall a lemma on dilations on the unitary dual $\widehat{G},$ which will be useful in our analysis of spectral multipliers.   For the proof, see Lemma 4.3 of \cite{FR}.
\begin{lemma}\label{dilationsrepre}
For every $\pi\in \widehat{G}$ let us define $D_{r}(\pi)=\pi^{(r)}$ by $D_{r}(\pi)(x)=\pi(rx)$ for every $r>0$ and $x\in G.$ Then, if $f\in L^{\infty}(\mathbb{R})$ then $f(\pi^{(r)}(\mathcal{R}))=f({r^{\nu}\pi(\mathcal{R})}).$
\end{lemma}

We refer to \cite[Chapter 4]{FR2} and references therein  for an exposition of further properties of Rockland operators and their history, and to ter Elst and Robinson \cite{TElst+Robinson} for their spectral properties.

\subsection{Sobolev spaces and the H\"ormander-Mihlin theorem}
In order to define Sobolev spaces, we choose  a positive left-invariant Rockland operator $\mathcal{R}$ of homogeneous degree $\nu>0$.  With notations above one defines  Sobolev spaces as follows (c.f \cite{FR2}).
\begin{definition}
Let $r\in \mathbb{R},$ the homogeneous Sobolev space $\dot{H}^{r,p}(G)$ consists of those $f\in \mathcal{D}'(G)$ satisfying
\begin{equation}
\Vert f\Vert_{\dot{H}^{r,p}(G)}:=\Vert \mathcal{R}^{\frac{r}{\nu}}f \Vert_{L^p(G)}<\infty.
\end{equation}
Analogously, the inhomogeneous Sobolev space ${H}^{r,p}(G)$ consists of those distributions $f\in \mathcal{D}'(G)$ satisfying 
\begin{equation}
\Vert f\Vert_{H^{r,p}(G)}:=\Vert (I+ \mathcal{R})^{\frac{r}{\nu}}f \Vert_{L^p(G)}<\infty.
\end{equation}
\end{definition}
By using a quasi-norm  $|\cdot| $ on $G$ we can introduce  for every $r\geq 0,$ the inhomogeneous Sobolev space of order $r$ on $\widehat{G},$ $H^{r}(\widehat{G})$ which is defined by $$H^{r}(\widehat{G})=\mathscr{F}_{G}(L^{2}(G, (1+|\cdot|^{2})^{\frac{r}{2}}dx)),$$ where $\mathscr{F}_{G}$ is the Fourier transform on the group $G.$  In a similar way, for $r\geq0$ the homogeneous Sobolev space $\dot{H}^{r}(\widehat{G})$ is defined by
$$\dot{H}^{r}(\widehat{G})=\mathscr{F}_{G}(L^{2}(G, |\cdot|^{r}dx)).$$
As usual if $r=0$ we denote $L^{2}(\widehat{G})=\dot{H}^{0}(\widehat{G})=H^{0}(\widehat{G}).$ Characterisations of Sobolev spaces on $G$ and on the unitary dual $\widehat{G}$ in terms of homogeneous norms on $G$ can be found in \cite{FR} and \cite{FR2}, respectively. 

Finally we present the H\"ormander-Mihlin theorem for graded nilpotent Lie groups. This theorem will be useful in our proof of the Littlewood-Paley theorem. The formulation of such result  requires a local notion of Sobolev space on the dual space $\widehat{G}.$  We introduce this as follows. Let $s\geq 0,$  we say that the field $\sigma=\{ \sigma(\pi):\pi\in\widehat{G}\}$ is locally uniformly in right-$H^{s}(\widehat{G})$ (resp. left-$H^{s}(\widehat{G})$)  if there exists a positive Rockland operator $\mathcal{R}$ and a function $\eta\in \mathcal{D}(G)$ satisfying
\begin{equation}
\Vert \sigma \Vert_{H^{s},l.u, R,\eta, \mathcal{R} }:=\sup_{r>0}\Vert\{\sigma(\pi^{(r)})\eta(\pi(\mathcal{R}))   \} \Vert_{H^{s}(\widehat{G})}<\infty,
\end{equation}
respectively,
\begin{equation}
\Vert \sigma \Vert_{H^{s},l.u, L,\eta, \mathcal{R} }:=\sup_{r>0}\Vert \{\eta(\pi(\mathcal{R}))\sigma(\pi^{(r)})   \} \Vert_{H^{s}(\widehat{G})}<\infty. 
\end{equation}
It important to mention that if $\phi$ is another function in $\mathcal{D}(0,\infty)$ then (see \cite{FR})
\begin{equation}\label{independencelu}
\Vert \sigma \Vert_{H^{s},l.u, R,\eta, \mathcal{R} }\asymp \Vert \sigma \Vert_{H^{s},l.u, R,\phi, \mathcal{R} },\textnormal{      and      } \Vert \sigma \Vert_{H^{s},l.u, L,\eta, \mathcal{R} }\asymp \Vert \sigma \Vert_{H^{s},l.u, L,\phi, \mathcal{R} }.
\end{equation}
The following lemma shows how Sobolev spaces on the unitary dual interact with the family of dilations.
  \begin{lemma}\label{lemmaLP}
Let $\sigma\in L^{2}(\widehat{G}).$ If $r>0$ and $s\geq 0$ then 
\begin{equation}
\Vert \sigma \circ D_{r}\Vert_{\dot{H}^s(\widehat{G})}= r^{s-\frac{Q}{2}}\Vert \sigma \Vert_{\dot{H}^{s}(\widehat{G})}.
\end{equation}
This implies that $\sigma \in \dot{H}^{s}(\widehat{G})$ if only if for every $r>0,$ $\sigma\circ D_{r}\in \dot{H}^{s}(\widehat{G}).$ Also, if $\mathcal{R},\mathcal{S}$ are positive Rockland operators and $\eta,\zeta\in \mathcal{D}(0,\infty),$ $\eta,\zeta\neq 0,$ then there exists $C>0$ such that
\begin{equation}
\Vert \sigma \Vert_{H^s,l.u,L,\zeta,\mathcal{S}}\leq C \Vert \sigma \Vert_{H^s,l.u,L,\eta,\mathcal{R}}
\end{equation}
and 
\begin{equation}
\Vert \sigma \Vert_{H^s,l.u,R,\zeta,\mathcal{S}}\leq C \Vert \sigma \Vert_{H^s,l.u,R,\eta,\mathcal{R}}.
\end{equation}
\end{lemma}
\begin{proof} By Lemma 2.1 or Lemma 4.3 of \cite{FR} we have
\begin{align*}\Vert \sigma\circ D_{r} \Vert_{\dot{H}^{s}(\widehat{G})} &=\Vert |\cdot|^{s} \mathscr{F}_{G}^{-1}(\sigma \circ D_{r}) \Vert_{L^{2}({G})}=\Vert |\cdot|^{s} r^{-Q}\mathscr{F}_{G}^{-1}(\sigma )(r^{-1}\cdot) \Vert_{L^{2}({G})}\\
&=r^{-\frac{Q}{2}}\Vert |r\cdot|^{s} \mathscr{F}_{G}^{-1}(\sigma ) \Vert_{L^{2}({G})}\\
&=r^{s-\frac{Q}{2}}\Vert \sigma\Vert_{\dot{H}^{s}(\widehat{G})}.
\end{align*}
With the equality above, it is clear that  $\sigma \in \dot{H}^{s}(\widehat{G})$ if only if for every $r>0,$ $\sigma\circ D_{r}\in \dot{H}^{s}(\widehat{G}).$  The second part of the Lemma  has been shown in Proposition 4.6 of \cite{FR}.
\end{proof}

Now, we state the H\"ormander-Mihlin theorem on the graded nilpotent Lie group $G$ (c.f. Theorem 4.11 of \cite{FR}):
\begin{theorem}\label{HMT}
Let $G$ be a graded Lie group. Let $\sigma\in L^{2}(\widehat{G}).$ If 
\begin{equation}\label{LPCRUZFIS}
    \Vert \sigma \Vert_{H^{s},l.u, L,\eta, \mathcal{R} },\Vert \sigma \Vert_{H^{s},l.u, R,\eta, \mathcal{R} }<\infty,
\end{equation} with $s>\frac{Q}{2},$ then the corresponding multiplier $T_{\sigma}$ extends to a bounded operator on $L^{p}(G)$ for all $1<p<\infty.$ Moreover
\begin{equation}
\Vert T_{\sigma} \Vert_{ \mathcal{L}(L^p(G))    }\leq C  \max\{ \Vert \sigma \Vert_{H^{s},l.u, L,\eta, \mathcal{R} },\Vert \sigma \Vert_{H^{s},l.u, R,\eta, \mathcal{R} } \}.
\end{equation}
\end{theorem}
The following remarks will be useful in our formulation of the Littlewood-Paley theorem.
\begin{remark}[On the proof of the H\"ormander-Mihlin Theorem]\label{theremarkof}The proof of the H\"ormander-Mihlin theorem (c.f. Theorem 4.11 of \cite{FR}) on graded Lie groups uses a suitable Littlewood-Paley decompostions of the symbol. Indeed,  for $\sigma$ satisfying \eqref{LPCRUZFIS}, the $L^p$-boundedness of $T_\sigma$ is proved in Theorem 4.11 of \cite{FR}, by decomposing 
\begin{equation}
    T_\sigma=\sum_{j\geq 0}T_{j},\,\,\, T_{j}:=T_\sigma \psi_{j}(\mathcal{R}),
\end{equation} and using that the right-convolution  kernels  of the family $T_j,$ $k_{j}:=\mathscr{F}^{-1}{(\sigma_{T_j})},$ summed on $j,$ provide the  distributional kernel of $T,$ $k=\sum_{j}k_{j},$ which agrees with a locally integrable function on $G\setminus \{0\},$ such that, for every $c>0,$
\begin{equation}\label{EQestimateweak(1,1)}
    \mathscr{I}_\ell:=\sup_{z\in G}\int\limits_{ |x|>4c|z|}|2^{-\ell Q}\kappa_{\ell}(2^{-\ell}\cdot z^{-1} x)  - 2^{-\ell Q}\kappa_{\ell}(2^{-\ell }\cdot x)|dx,
\end{equation} satisfies (see \cite{FR}, p. 26), $\mathscr{I}_\ell\lesssim    2^{-\ell \varepsilon_0}\max\{\Vert \sigma \Vert_{H^{s},l.u, L,\eta, \mathcal{R} },\Vert \sigma \Vert_{H^{s},l.u, R,\eta, \mathcal{R} }\},$ for some $\varepsilon_0>0,$ depending only of $c>0.$
The proof of the H\"ormander-Mihlin theorem developed by V. Fischer and the second author consists of proving that this kernel estimates are sufficiently good in order that  $$\Vert T_j\Vert_{\mathscr{B}(L^p(G))}\leq \mathscr{I}_j   \max\{\Vert \sigma \Vert_{H^{s},l.u, L,\eta, \mathcal{R} },\Vert \sigma \Vert_{H^{s},l.u, R,\eta, \mathcal{R} }\}$$ and consequently 
$$\Vert T\Vert_{\mathscr{B}(L^p(G))}\lesssim \sum_j  2^{-j \varepsilon_0}\max\{\Vert \sigma \Vert_{H^{s},l.u, L,\eta, \mathcal{R} },\Vert \sigma \Vert_{H^{s},l.u, R,\eta, \mathcal{R} }\}. $$

In particular, if $T=I,$ is the identity operator on $L^p(G),$ $\sigma(\pi)=I_{H_{\pi}},$ is the identity operator on $H_\pi,$ and the right convolution kernel $\kappa_\ell$ associated with  $\psi_{\ell}(\mathcal{R}),$ satisfies the estimate
\begin{equation}\label{ThLiWeak}
    \mathscr{I}_\ell\lesssim  2^{-\ell\varepsilon_0}.
\end{equation} The inequality \eqref{ThLiWeak} will be useful in our proof in the Littlewood-Paley theorem.
\end{remark}

In the next sections, we present our main results. We start with a formulation of the Nikolskii inequality on graded Lie groups.  

\section{Nikolskii Inequality on graded  Lie groups}\label{Nikolskii}

Let $G$ be a graded Lie group with a family of dilations $D_{t},$ $t>0.$ Let $\mathcal{R}$ be a positive Rockland operator of homogeneous degree $\nu>0,$ and for every $L>0,$ let us consider the linear operator  $\psi_{L}(\mathcal{R}),$ defined by the functional calculus, where $\psi_{L}(t)=\psi(L^{-1}t)$ and $\psi\in\mathscr{D}(0,\infty)$ is a function with compact   support in $[\frac{1}{2},2].$  In terms of the spectral resolution $(E(\lambda))_{\lambda\geq 0}$ associated with $\mathcal{R},$ we have
\begin{equation}
    T_{L}f\equiv \psi_{L}(\mathcal{R})f:=\int_{0}^\infty\psi(L^{-1}\lambda)dE(\lambda)f,\,\,
\end{equation} for every $f\in \mathscr{S}(G).$
Then $T_{L}$ is a spectral multiplier and $$\mathscr{F}_{G}(T_{L}f)(\pi)=\left(\int_{0}^{\infty} \psi_{L   }(\lambda)dE_{\pi}(\lambda)\right)\widehat{f}(\pi),$$ where $(E_\pi(\lambda))_{\lambda\geq 0}$ is the spectral resolution of $\pi(\mathcal{R}):=d\pi(\mathcal{R}).$ We denote $E_{\pi}(L)=\int_{0}^{\infty} \psi(L^{-1}\lambda)dE_{\pi}(\lambda),$ where $\psi\in \mathscr{D}(0,\infty)$ is a function satisfying $\psi=1$ on $[\frac{1}{2},1].$ In terms of the Fourier inversion formula we have
\begin{equation}\label{tl}
T_{L}f(x)=\int_{\widehat{G}}\textrm{Tr}[\pi(x)E_{\pi}(L)\psi_{L}(\pi)\widehat{f}(\pi)]d\pi.
\end{equation}
With notations above we present our version of the Nikolskii inequality in the following theorem.
\begin{theorem}\label{Nikolskii}
Let $G$ be a graded Lie group of homogeneous dimension $Q$, and let us consider the operator $T_{L}$ as in \eqref{tl}. If $1\leq p\leq q\leq \infty$ then
\begin{equation}
\Vert T_{L}f \Vert_{L^q}\leq  \Vert \mathscr{F}_{G}^{-1}[E_{\pi}(1)] \Vert_{L^{r}} L^{\frac{Q}{\nu}(\frac{1}{p}-\frac{1}{q})}\Vert T_Lf \Vert_{L^{p}},
\end{equation}
where $r=(1+(1/q-1/p))^{-1}.$
Since $\mathcal{F}_{G}^{-1}[E_{\pi}(1)]\in\mathscr{S}(G),$ its $L^r$-norm is finite.
\end{theorem}
\begin{proof}
Let us define for every $L>0,$ the function $$g_{L}:=L^{-{\frac{Q}{\nu}}}(T_Lf)\circ D_{L^{-\frac{1}{\nu}}},$$ i.e., $g_{L}(x)=L^{-{\frac{Q}{\nu}}}(T_Lf)(L^{-{\frac{1}{\nu}}}x),$ $x\in G.$  Denoting by $e=e_{G}$  the identity element of $G$, for every $\pi\in \widehat{G}$ we have
\begin{align*}
\widehat{g}_{L}(\pi)&=\int_{G}L^{-{\frac{Q}{\nu}} }(T_Lf)(L^{-{\frac{1}{\nu}}}x)\pi(x)^*dx\\
&=\int_{G}(T_Lf)(y)\pi(L^{\frac{1}{\nu}}y)^*dy\\
&=\widehat{T_{L}f}(\pi^{(L^{\frac{1}{\nu}})}).
\end{align*}
We observe that
$$ \widehat{T_{L}f}(\pi^{(L^{\frac{1}{\nu}})})=E_{\pi^{(L^{\frac{1}{\nu}})}}(L)\psi_{L}(\pi^{(L^{\frac{1}{\nu}})})\widehat{f}(\pi^{(L^{\frac{1}{\nu}})}).$$
By using the fact that for every $a\in \mathbb{R},$ $\pi^{(a)}(\mathcal{R})=\pi(a^{\nu}\mathcal{R}),$ in particular with $a=L^\frac{1}{\nu},$ (see Lemma \ref{dilationsrepre}, or \cite[Lemma 4.3]{FR}) we have
$$ E_{\pi^{(L^{\frac{1}{\nu}})}}(L)\psi_{L}(\pi^{(L^{\frac{1}{\nu}})})\widehat{f}(\pi^{(L^{\frac{1}{\nu}})})=\phi_{L}(L\pi(\mathcal{R}))\psi_{L}(L\pi(\mathcal{R}))\widehat{f}(\pi^{(L^{\frac{1}{\nu}})}), $$
and by considering that
\begin{align*}
\phi_{L}(L\pi(\mathcal{R}))\psi_{L}(L\pi(\mathcal{R}))&=\int_{0}^{\infty}\phi_{L}(L\lambda)\psi_{L}(L\lambda)dE_{\pi}(\lambda)=\int_{0}^{\infty}\phi(\lambda)\psi(\lambda)dE_{\pi}(\lambda),
\end{align*}  
we obtain that $\widehat{g}_{L}(\pi)=E_{\pi}(1)\widehat{f}(\pi^{(L^{\frac{1}{\nu}} )}).$ By properties of the functional calculus, we have 

$$\widehat{g}_{L}(\pi)=E_{\pi}(1)\widehat{g}_{L}(\pi).$$
Hence
$$g_{L}(x)=g_{L}\ast \mathscr{F}_{G}^{-1}[E_{\pi}(1)](x),\,\, x\in G.  $$
By applying Young inequality we have
\begin{equation}\label{gl}
\Vert g_{L} \Vert_{L^{q}}\leq \Vert \mathscr{F}_{G}^{-1}[E_{\pi}(1)] \Vert_{L^{r}}\Vert g_{L} \Vert_{L^{p}},
\end{equation}
provided that $\frac{1}{p}+\frac{1}{r}=\frac{1}{q}+1.$ We observe that the condition $1\leq p\leq q\leq \infty$ implies that $0\leq \frac{1}{r}=1+\frac{1}{q}-\frac{1}{p}\leq 1$ and consequently $1\leq r\leq \infty.$  Observe that for every $a>0$ we have
\begin{align*}
\Vert g_{L} \Vert_{L^a(G)}&=\left( \int_{G}|g_{L}(x)|^{a}dx\right)^{\frac{1}{a}}=\left( \int_{G}L^{-\frac{Qa}{\nu}}|T_{L}f(L^{-\frac{1}{\nu}}x)|^{a}dx\right)^{\frac{1}{a}}\\
&=\left( \int_{G}L^{[Q-Qa]/\nu}|T_{L}f(y)|^{a}dy\right)^{\frac{1}{a}}\\
&=L^{\frac{Q}{\nu}(\frac{1}{a}-1)}\Vert T_L f\Vert_{L^{a}}.
\end{align*}
So, by the inequality \eqref{gl}, we have
\begin{equation}
L^{\frac{Q}{\nu}(\frac{1}{q}-1)}\Vert T_L f \Vert_{L^{q}}\leq \Vert \mathscr{F}_{G}^{-1}[E_{\pi}(1)] \Vert_{L^{r}} L^{\frac{Q}{\nu}(\frac{1}{p}-1)}\Vert T_Lf \Vert_{L^{p}}.
\end{equation}
Thus, we obtain

\begin{equation}
\Vert T_{L}f \Vert_{L^q}\leq  \Vert \mathscr{F}_{G}^{-1}[E_{\pi}(1)] \Vert_{L^{r}} L^{\frac{Q}{\nu}(\frac{1}{p}-\frac{1}{q})}\Vert T_Lf \Vert_{L^{p}}.
\end{equation}
This completes the proof.
\end{proof}

\section{A vector valued-inequality for Littlewood-Paley decompositions and the Littlewood-Paley theorem on graded Lie groups}\label{Littlewood-Paley}

The Littlewood-Paley theory provides a partial substitute in $L^p$ spaces for the results
derived from the Plancherel theorem. The main notion in the Littlewood-Paley theory is the concept of a dyadic decomposition.  Here, the sequence $\{\psi_{l}\}_{l\in\mathbb{N}_{0}}$ is a dyadic decomposition,  defined as follows:  we choose a function $\psi_0\in C^{\infty}_{0}(\mathbb{R}),$  $\psi_0(\lambda)=1,$  if $|\lambda|\leq 1,$ and $\psi_0(\lambda)=0,$ for $|\lambda|\geq 2.$ For every $j\geq 1,$ let us define $\psi_{j}(\lambda)=\psi_{0}(2^{-j}\lambda)-\psi_{0}(2^{-j+1}\lambda).$ For $\psi(\lambda):=\psi_0(\lambda)-\psi_{0}(2\lambda),$ $\psi_{j}(\lambda)=\psi(2^{-j}\lambda).$  In particular, we have
\begin{eqnarray}\label{deco1}
\sum_{l\in\mathbb{N}_{0}}\psi_{l}(\lambda)=1,\,\,\, \text{for every}\,\,\, \lambda>0.
\end{eqnarray}
For versions of the Littlewood-Paley theorem for the sub-Laplacian on the Heisenberg group we can refer to Bahouri, G\'erard and Xu \cite{bahouri}, and for sub-Laplacians on groups of polynomial growth see Furioli, Melzi and Veneruso \cite{furioli}. Here we prove it for general Rockland operators on graded groups. Now we present the Littlewood-Paley theorem in the form of the following result.
\begin{theorem}\label{LPT}
Let  $1<p<\infty$ and let $G$ be a graded Lie group. If $\mathcal{R}$ is a positive Rockland operator then there exist  constants $0<c_p,C_{p}<\infty$ depending only on $p$ and $\psi_0$ such that
\begin{equation}\label{LPTequ}
c_p\Vert f\Vert_{L^{p}(G)}\leq \left\Vert \left(\sum_{l=0}^{\infty} |\psi_{l}(\mathcal{R})f|^{2}    \right)^{\frac{1}{2}}\right\Vert_{L^{p}(G)}\leq  C_{p}\Vert f\Vert_{L^{p}},
\end{equation}
holds for every $f\in L^{p}(G).$ Moreover, for $p=1,$  there exists a constant $C>0$ independent of $f\in L^1(G)$ and $t>0,$ such that 
\begin{equation}\label{weak(1,1)inequality}
    \left|\left\{x\in G:\left(\sum_{\ell=0}^\infty|\psi_\ell(\mathcal{R}) f(x)|^2   \right)^{\frac{1}{2}}>t \right\}\right|\leq \frac{C}{t}\Vert f\Vert_{L^1(G)}.
\end{equation}
\end{theorem}

For the proof of Theorem \ref{LPT}, we will assume for a moment the following  theorem. 

\begin{theorem}\label{Theoremr}
Let  $1<p,r<\infty$ and let $G$ be a graded Lie group. If $\mathcal{R}$ is a positive Rockland operator then there exist  constants $C_p>0$ depending only on $p$ and $\psi_0,$ such that
\begin{equation}
 \left\Vert \left(\sum_{\ell=0}^{\infty} |\psi_{\ell}(\mathcal{R})f_\ell|^{r}    \right)^{\frac{1}{r}}\right\Vert_{L^{p}(G)}\leq C_p \left\Vert \left(\sum_{\ell=0}^{\infty} |f_\ell(x)|^{r}    \right)^{\frac{1}{r}}\right\Vert_{L^{p}(G)}=: C_{p}\Vert \{f_\ell\}\Vert_{L^{p}(G,\ell^r(\mathbb{N}_0^n))}.
\end{equation}Moreover, for $p=1,$  there exists a constant $C>0$ independent of $\{f_\ell\}\in L^1(G,\ell^r(\mathbb{N}_0))$ and $t>0,$ such that 
\begin{equation}
    \left|\left\{x\in G:\left(\sum_{\ell=0}^\infty|\psi_\ell(\mathcal{R}) f_\ell(x)|^r   \right)^{\frac{1}{r}}>t \right\}\right|\leq \frac{C}{t}\Vert \{f_\ell\} \Vert_{L^1(G,\ell^r(\mathbb{N}_0)}.
\end{equation}
\end{theorem}

\begin{proof}[Proof of Theorem \ref{LPT}]

First we will prove that for every positive function $f\in L^p(G)\cap L^1(G),$ the estimate 
\begin{equation}\label{DCMR}
    \left\Vert \left(\sum_{l=0}^{\infty} |\psi_{l}(\mathcal{R})f|^{2}    \right)^{\frac{1}{2}}\right\Vert_{L^{p}(G)}\leq  C_p\Vert f\Vert_{L^{p}(G)},
\end{equation} holds true for every $1<p<\infty,$ and the inequality in the right hand side of \eqref{DCMR} can be extended to general  $f\in L^p(G)$ by the density argument. We will employ an argument of interpolation. First, let  us  prove \eqref{weak(1,1)inequality}. Indeed, it is equivalent to the fact that the vector-valued operator
\begin{equation}
    W(f)=\{\psi_\ell(\mathcal{R})f\}_{\ell=0}^\infty,
\end{equation} admits a bounded extension from $L^1(G)$ into $L^{1,\infty}(G,\ell^2(\mathbb{N}_0^n)).$ In view of the almost orthogonality of the functions  $x\mapsto f_\ell(x):=\psi_\ell(\mathcal{R}) f(x),$ on $L^2(G),$ which is consequence of the following property on the supports of the functions $\psi_{\ell},$ 
\begin{equation}
    \textnormal{supp}(\psi_\ell)\cap  \textnormal{supp}(\psi_{\ell'})=\emptyset,\,\,\,\,\, |\ell-\ell'|\geq 2,
\end{equation}
 we have
\begin{align*}
    \int\limits_{G} \sum_{\ell=0}^\infty |\psi_\ell(\mathcal{R}) f(x)|^2dx &\asymp\sum_{\ell=0}^\infty\int\limits_{G}|\psi_\ell(\mathcal{R}) f(x)|^2dx \asymp\int\limits_{G}  |\sum_{\ell=0}^\infty \psi_\ell(\mathcal{R})f(x)|^2dx\\
    &=\int\limits_{G}  |f(x)|^2dx, 
\end{align*}  which implies that $W$ admits a bounded extension from $L^2(G)$ into $L^{2}(G,\ell^2(\mathbb{N}_0^n)).$ So, if we prove \eqref{weak(1,1)inequality},  interpolating with the $L^2(G)-L^{2}(G,\ell^2(\mathbb{N}_0^n))$-boundedness of $W,$ we obtain that $W$ extends to a bounded operator from  $L^p(G)$ into $L^{p}(G,\ell^2(\mathbb{N}_0^n)),$ for all $1<p\leq 2.$ We will then extend the boundedness of $W$ for all $2\leq p<\infty,$ by using the duality argument. So, our proof consists of the following steps.
\begin{itemize}
\item {\bf Step 0.} Assume that $f$ is a non-negative function in $L^p(G).$
    \item {\bf Step 1.} Prove  the weak $(1,1)$-inequality  \eqref{weak(1,1)inequality}.
    \item {\bf Step 2.} Interpolation between  \eqref{weak(1,1)inequality} and the boundedness of $W$ from $L^2(G)$ into $L^{2}(G,\ell^2(\mathbb{N}_0^n)),$ in order to prove \eqref{DCMR} for all $1<p\leq 2.$
    \item {\bf Step 3.} Apply the duality argument for extending  \eqref{DCMR} for all $2\leq p<\infty.$
     \item {\bf Step 4.} Proof of the left hand side of \eqref{LPTequ}.
     \item{\bf Step 5.} Extend  \eqref{LPTequ} and \eqref{weak(1,1)inequality}  to general real-valued  functions in $L^p(G)$.
     \item{\bf Step 6.}  Extend  \eqref{LPTequ} and \eqref{weak(1,1)inequality} to general complex-valued functions  in $L^p(G)$.
\end{itemize}
Step 1.  Apply the Calder\'on-Zygmund decomposition Lemma to the non-negative function $f\in L^p(G)\cap L^1(G)\subset L^1(G),$ under the identification $G\simeq \mathbb{R}^n,$ (see, e.g. Hebish \cite{Hebish}) in order to obtain a disjoint collection $\{I_j\}_{j=0}^{\infty}$ of disjoint open sets such that
\begin{itemize}
    \item $f(x)\leq t,$ for $a.e.$ $x\in G\setminus \cup_{j\geq 0}I_j,$\\
    
    \item $\sum_{j\geq 0}|I_j|\leq \frac{C}{t}\Vert f\Vert_{L^1(G)},$ and\\ 
    
    \item $t|I_j|\leq \int_{I_j}f(x)dx\leq 2|I_j|t,$ for all $j.$ 
\end{itemize} Moreover, for every $j,$ let us define $R_j$ by
\begin{equation}\label{Rj}
    R_{j}:=\sup\{R>0: B(z_j,R)\subset I_j, \textnormal{   for some   }z_j\in I_j\},
\end{equation} where $B(z_j,R)=\{x\in I_j:|z_j^{-1}x|<R\}.$ Then, we can assume that every $I_j$ is diffeomorphic to an open cube on $\mathbb{R}^n,$ that it is bounded, and that  $I_j\subset B(z_j,2R_j),$ where $z_j\in I_j$ (see Hebish \cite{Hebish}). 
\begin{remark}
Before of continuing with the proof note that by assuming $f(e_G)>t,$ (this just re-defining $f\in L^p(G)\cap L^1(G)$ at the identity element) we should have that
\begin{equation}\label{eGasump}
    e_{G}\in \bigcup_{j}I_j,
\end{equation} because $f(x)\leq t,$ for $a.e.$ $x\in G\setminus \cup_{j\geq 0}I_j.$
\end{remark}
Let us define, for every $x\in I_j,$
\begin{equation}
    g(x):=\frac{1}{|I_j|}\int\limits_{I_j}f(y)dy,\,\,\,b(x)=f(x)-g(x),
\end{equation} and for $x\in G\setminus\cup_{j\geq 0}I_j,$
\begin{equation}
    g(x)=f(x),\,\,\, b(x)=0.
\end{equation} Observe that for every $x\in I_j,$ 
\begin{align*}
    |g(x)|= \left|   \frac{1}{|I_j|}\int\limits_{I_j}f(y)dy \right|\leq 2t.
\end{align*}
In view of the Minkowski inequality, we deduce that
\begin{align*}
   & \left|\left\{x\in G:\left(\sum_{\ell=0}^\infty|\psi_\ell(\mathcal{R}) f(x)|^2   \right)^{\frac{1}{2}}>t \right\}\right|\leq \left|\left\{x\in G:\left(\sum_{\ell=0}^\infty|\psi_\ell(\mathcal{R}) g(x)|^2   \right)^{\frac{1}{2}}>\frac{t}{2} \right\}\right|\\
    &\hspace{3cm}+\left|\left\{x\in G:\left(\sum_{\ell=0}^\infty|\psi_\ell(\mathcal{R}) b(x)|^2   \right)^{\frac{1}{2}}>\frac{t}{2} \right\}\right|.
\end{align*}
By the Chebyshev inequality, we have
\begin{align*}
    & \left|\left\{x\in G:\left(\sum_{\ell=0}^\infty|\psi_\ell(\mathcal{R}) f(x)|^2   \right)^{\frac{1}{2}}>t \right\}\right|\\
    &\leq  \left|\left\{x\in G:\left(\sum_{\ell=0}^\infty|\psi_\ell(\mathcal{R}) g(x)|^2   \right)^{\frac{1}{2}}>\frac{t}{2} \right\}\right|+ \left|\left\{x\in G:\left(\sum_{\ell=0}^\infty|\psi_\ell(\mathcal{R}) b(x)|^2   \right)^{\frac{1}{2}}>\frac{t}{2} \right\}\right|\\
    &= \left|\left\{x\in G:\sum_{\ell=0}^\infty|\psi_\ell(\mathcal{R}) g(x)|^2   >\frac{t^2}{2^2} \right\}\right|+ \left|\left\{x\in G:\left(\sum_{\ell=0}^\infty|\psi_\ell(\mathcal{R}) b(x)|^2   \right)^{\frac{1}{2}}>\frac{t}{2} \right\}\right|\\
    &\leq \frac{2^2}{t^2}\int_{G}\sum_{\ell=0}^\infty|\psi_\ell(\mathcal{R}) g(x)|^2dx+ \left|\left\{x\in G:\left(\sum_{\ell=0}^\infty|\psi_\ell(\mathcal{R}) b(x)|^2   \right)^{\frac{1}{2}}>\frac{t}{2} \right\}\right|.
\end{align*}
In view of the almost orthogonality of the functions   $x\mapsto g_\ell(x):=\psi_\ell(\mathcal{R}) g(x),$ on $L^2(G),$ we have
\begin{align*}
    \int\limits_{G} \sum_{\ell=0}^\infty |\psi_\ell(\mathcal{R}) g(x)|^2dx &=\sum_{\ell=0}^\infty\int\limits_{G}|\psi_\ell(\mathcal{R}) g(x)|^2dx &\asymp \int\limits_{G}  |\sum_{\ell=0}^\infty \psi_\ell(\mathcal{R})g(x)|^2dx=\int\limits_{G}  |g(x)|^2dx.
\end{align*} The  estimate
\begin{align*}
    \Vert g\Vert_{L^2(G)}^2&=\int\limits_{G}|g(x)|^2dx=\sum_{j}\int\limits_{I_j}|g(x)|^2dx+\int\limits_{G\setminus \cup_{j}I_j}|g(x)|^2dx\\
    &=\sum_{j}\int\limits_{I_j}|g(x)|^2dx+\int\limits_{G\setminus \cup_{j}I_j}|f(x)|^2dx\\
    &\leq \sum_{j}\int\limits_{I_j}(2t)^{2}dx+\int\limits_{G\setminus \cup_{j}I_j}f(x)^2dx\lesssim t^{2}\sum_{j}|I_j|+\int\limits_{G\setminus \cup_{j}I_j}f(x)f(x)dx\\
    &\leq t^{2}\times \frac{C}{t}\Vert f\Vert_{L^1(G)}+t\int\limits_{G\setminus \cup_{j}I_j}f(x)dx\lesssim t\Vert f\Vert_{L^1(G)},
\end{align*} implies that,
\begin{align*}
  \Small{  \left|\left\{x\in G:\left(\sum_{\ell=0}^\infty|\psi_\ell(\mathcal{R}) f(x)|^2   \right)^{\frac{1}{2}}>t \right\}\right|\leq \frac{4}{t}\Vert f\Vert_{L^1(G)}+ \left|\left\{x\in G:\left(\sum_{\ell=0}^\infty|\psi_\ell(\mathcal{R}) b(x)|^2   \right)^{\frac{1}{2}}>\frac{t}{2} \right\}\right|.}
\end{align*}
Taking into account that $b\equiv 0$ on $G\setminus \cup_j I_j,$ we have that
\begin{equation}
    b=\sum_{k}b_k,\,\,\,b_k(x)=b(x)\cdot 1_{I_k}(x).
\end{equation} Let us assume that $I_{j}^*$ is a open set, such that $I_j\subset I_j^*,$ and $|I_{j}^*|=K|I_{j}|$ for some $K>0,$ and $\textnormal{dist}(\partial I_{j}^*,\partial I_{j})\geq 4c\,\textnormal{dist}(\partial I_{j},e_{G}),$ where $c>0$ and $e_{G}$ is the identity element of $G$.  So, by the Minkowski inequality we have,
\begin{align*}
     & \left|\left\{x\in G:\left(\sum_{\ell=0}^\infty|\psi_\ell(\mathcal{R}) b(x)|^2   \right)^{\frac{1}{2}}>\frac{t}{2} \right\}\right|\\
      &=\left|\left\{x\in \cup_j I_j^*:\left(\sum_{\ell=0}^\infty|\psi_\ell(\mathcal{R}) b(x)|^2   \right)^{\frac{1}{2}}>\frac{t}{2} \right\}\right|+\left|\left\{x\in G\setminus  \cup_j I_j^*:\left(\sum_{\ell=0}^\infty|\psi_\ell(\mathcal{R}) b(x)|^2   \right)^{\frac{1}{2}}>\frac{t}{2} \right\}\right|\\
      &\leq \left|\left\{x\in G:x\in \cup_j I_j^* \right\}\right|+\left|\left\{x\in G\setminus  \cup_j I_j^*:\left(\sum_{\ell=0}^\infty|\psi_\ell(\mathcal{R}) b(x)|^2   \right)^{\frac{1}{2}}>\frac{t}{2} \right\}\right|.
      \end{align*} Since $$  \left|\left\{x\in G:x\in \cup_j I_j^* \right\}\right| \leq \sum_{j}|I_j^*|,  $$ we have
      \begin{align*}  & \left|\left\{x\in G:\left(\sum_{\ell=0}^\infty|\psi_\ell(\mathcal{R}) b(x)|^2   \right)^{\frac{1}{2}}>\frac{t}{2} \right\}\right|\\
      &\leq\sum_{j}|I_j^*|+\left|\left\{x\in G\setminus  \cup_j I_j^*:\left(\sum_{\ell=0}^\infty|\psi_\ell(\mathcal{R}) b(x)|^2   \right)^{\frac{1}{2}}>\frac{t}{2} \right\}\right|\\
      &=K\sum_{j}|I_j|+\left|\left\{x\in G\setminus  \cup_j I_j^*:\left(\sum_{\ell=0}^\infty|\psi_\ell(\mathcal{R}) b(x)|^2   \right)^{\frac{1}{2}}>\frac{t}{2} \right\}\right|\\
      &\leq \frac{CK}{t}\Vert f\Vert_{L^1(G)}+\left|\left\{x\in G\setminus  \cup_j I_j^*:\left(\sum_{\ell=0}^\infty|\psi_\ell(\mathcal{R}) b(x)|^2   \right)^{\frac{1}{2}}>\frac{t}{2} \right\}\right|.
  \end{align*}The Chebyshev inequality implies that
  \begin{align*}
     & \left|\left\{x\in G\setminus  \cup_j I_j^*:\left(\sum_{\ell=0}^\infty|\psi_\ell(\mathcal{R}) b(x)|^2   \right)^{\frac{1}{2}}>\frac{t}{2} \right\}\right|\\
     &\leq\frac{2}{t}\int\limits_{ G\setminus  \cup_j I_j^*} \left(\sum_{\ell=0}^\infty|\psi_\ell(\mathcal{R}) b(x)|^2   \right)^{\frac{1}{2}}dx\\
     &=\frac{2}{t}\int\limits_{ G\setminus  \cup_j I_j^*} \left(\sum_{\ell=0}^\infty\left|\left(\psi_\ell(\mathcal{R})\left(\sum_{k} b_k\right)   \right)(x)\right|^2\right)^{\frac{1}{2}}dx\\
     &=\frac{2}{t}\int\limits_{ G\setminus  \cup_j I_j^*} \left\Vert \{(\sum_{k} \psi_\ell(\mathcal{R})b_k   (x))\}_{\ell=0}^\infty\right\Vert_{\ell^2(\mathbb{N}_0)}dx\\
      &\leq \frac{2}{t}\int\limits_{ G\setminus  \cup_j I_j^*} \sum_{k}\left\Vert \{( \psi_\ell(\mathcal{R})b_k   (x))\}_{\ell=0}^\infty\right\Vert_{\ell^2(\mathbb{N}_0)}dx\\
     &= \frac{2}{t}\sum_{k}\int\limits_{ G\setminus  \cup_j I_j^*} \left(\sum_{\ell=0}^\infty\left|\left(\psi_\ell(\mathcal{R})b_k   \right)(x)\right|^2\right)^{\frac{1}{2}}dx.
  \end{align*}
If $\kappa_\ell$ is the right convolution kernel of     $\psi_\ell(\mathcal{R}) ,$ from the inequality, 
\begin{equation}
    \left(\sum_{\ell=0}^\infty\left|\left(\psi_\ell(\mathcal{R})b_k   \right)(x)\right|^2\right)^{\frac{1}{2}}\leq \sum_{\ell=0}^\infty\left|\left(\psi_\ell(\mathcal{R})b_k   \right)(x)\right|
\end{equation}  we deduce
\begin{align*}
     & \left|\left\{x\in G\setminus  \cup_j I_j^*:\left(\sum_{\ell=0}^\infty|\psi_\ell(\mathcal{R}) b(x)|^2   \right)^{\frac{1}{2}}>\frac{t}{2} \right\}\right|\leq  \frac{2}{t}\sum_{k}\int\limits_{ G\setminus  \cup_j I_j^*}\sum_{\ell=0}^\infty\left|\left(\psi_\ell(\mathcal{R})b_k   \right)(x)\right|dx\\
     &= \frac{2}{t}\sum_{k}\int\limits_{ G\setminus  \cup_j I_j^*}\sum_{\ell=0}^\infty\left|b_k   \ast \kappa_{\ell}(x)\right|dx\\
     &= \frac{2}{t}\sum_{k}\int\limits_{ G\setminus  \cup_j I_j^*}\sum_{\ell=0}^\infty\left|\int\limits_{I_k}b_k(z)\kappa_{\ell}(z^{-1}x)dz   \right|dx.
     \end{align*} By using that  $\int_{I_k}b_{k}(z)dz=0,$ we have
     \begin{align*}
     \frac{2}{t}\sum_{k} \int\limits_{ G\setminus  \cup_j I_j^*}  &\sum_{\ell=0}^\infty\left|\int\limits_{I_k}b_k(z)\kappa_{\ell}(z^{-1}x)dz   \right|dx\\
      &=\frac{2}{t}\sum_{k}\int\limits_{ G\setminus  \cup_j I_j^*}\sum_{\ell=0}^\infty\left|\int\limits_{I_k}b_k(z)\kappa_{\ell}(z^{-1}x)dz  - \kappa_{\ell}(x)\int\limits_{I_k}b_{k}(z)dz \right|dx\\
       &=\frac{2}{t}\sum_{k}\int\limits_{ G\setminus  \cup_j I_j^*}\sum_{\ell=0}^\infty\left|\int\limits_{I_k}(\kappa_{\ell}(z^{-1}x)  - \kappa_{\ell}(x))b_{k}(z) dz\right|dx\\
      &\leq\frac{2}{t}\sum_{k}\int\limits_{ G\setminus  \cup_j I_j^*} \sum_{\ell=0}^\infty\int\limits_{I_k}|(\kappa_{\ell}(z^{-1}x)  - \kappa_{\ell}(x))b_{k}(z)|dzdx\\
      &=\frac{2}{t}\sum_{k}\int\limits_{I_k}   \sum_{\ell=0}^\infty\int\limits_{ G\setminus  \cup_j I_j^*}|\kappa_{\ell}(z^{-1}x)  - \kappa_{\ell}(x)|dx|b_{k}(z)|dz.
     \end{align*}
 If we assume for a moment that
 \begin{equation}\label{calderonkernels}
    M=\sup_{k} \sup_{z\in I_k}\sum_{\ell=0}^\infty\int\limits_{ G\setminus  \cup_j I_j^*}|\kappa_{\ell}(z^{-1}x)  - \kappa_{\ell}(x)|dx<\infty,
 \end{equation}then we have
     \begin{align*}
          & \left|\left\{x\in G\setminus  \cup_j I_j^*:\left(\sum_{\ell=0}^\infty|\psi_\ell(\mathcal{R}) b(x)|^2   \right)^{\frac{1}{2}}>\frac{t}{2} \right\}\right|\leq \frac{2M}{t}\sum_{k}\int\limits_{I_k}|b_{k}(z)|dz\\
          &=\frac{2M}{t}\| b\|_{L^1(G)}\\
          &\leq \frac{6M}{t}\| f\|_{L^1(G)}.
     \end{align*}
     So, if we prove \eqref{calderonkernels} we obtain the weak (1,1) inequality \eqref{weak(1,1)inequality} and we finish the first step of the proof. The proof of \eqref{calderonkernels} lies in the estimates of the Calder\'on-Zygmund kernel of every operator $\psi_{\ell}(\mathcal{R}).$ Because $\textnormal{dist}(\partial I_{j}^*,\partial I_{j})\geq 4c\textnormal{dist}(\partial I_{j},e_{G}),$ for $x\in G\setminus \cup_jI_j^*,$   for $z\in I_{k},$ $4c|z|=4c\times \textnormal{dist}(z,e_G)\lesssim  \textnormal{dist}(\partial I_{k}^*,\partial I_{k})  \leq|x|.$  Indeed, fix $\varepsilon>0,$ and let us take 
     $w\in \partial I_{k},$ and $w'\in \partial I_{k}^*$ such that $d(w,w')\leq \textnormal{dist}(\partial I_{k},\partial I_{k}^*)+\varepsilon. $ Then, from the triangle inequality, we have
    \begin{equation}\label{Estimateforboundaries} 
\begin{array}{l}
        d(z,e_G)\\
        \leq d(z,w)+d(w,w')+d(w',e_G)
         \leq \textnormal{diam}(I_k)+\textnormal{dist}(\partial I_{k},\partial I_{k}^*)+\textnormal{dist}(\partial I_k^*,e_G)+\varepsilon\\
         \lesssim \textnormal{diam}(I_k)+\textnormal{dist}(\partial I_{k},\partial I_{k}^*)+\textnormal{dist}(\partial I_{k},e_G)+\varepsilon\\
         \lesssim \textnormal{diam}(I_k)+\textnormal{dist}(\partial I_{k},\partial I_{k}^*)+\frac{1}{4c}\textnormal{dist}(\partial I_{k},\partial I_{k}^*)+\varepsilon\\
         \asymp \textnormal{dist}(\partial I_{k},\partial I_{k}^*)+\varepsilon,
\end{array}
\end{equation}
where in the last line we have assumed that $\textnormal{diam}(I_k)\asymp \textnormal{dist}(\partial I_{k},\partial I_{k}^*),$ (with constants of proportionality independent in $k$) and that $\textnormal{dist}(\partial I_{k},\partial I_{k}^*)$ is proportional to $R_k$ in view of the relation $|I_k^*|=K|I_k|.$ Assuming \eqref{Estimateforboundaries}, one has that for all $\varepsilon>0,$ $ d(z,e_G)\lesssim  \textnormal{dist}(\partial I_{k},\partial I_{k}^*)+\varepsilon,$ which implies that \begin{equation}\label{mainestimateofk}
    d(z,e_G)\lesssim  \textnormal{dist}(\partial I_{k},\partial I_{k}^*).
\end{equation}  To show  that the proportionality constant in \eqref{mainestimateofk} is uniform in $k,$  let us recall the definition of the radii $R_{k}'s$ in \eqref{Rj},
  that  $B(z_k,R_{k})\subset I_k\subset B(z_{k},2R_{k}),$ and that $B(z_k,R_{k}/C)\subset I_{k}^{*}\subset B(z_k, CR_{k})$ for some $C>2$ independent of $k,$ where for any $k,$ $z_k\in I_{k}.$ From this remark observe that:
  \begin{itemize}
      \item The condition  $B(z_k,R_{k})\subset I_k\subset B(z_{k},2R_{k}),$ implies that $2R_k\leq \textnormal{diam}(I_k)\leq 4R_{k}. $
      \item That $B(z_k,R_{k})\subset I_k\subset I_{k}^{*}\subset B(z_{k},CR_{k}),$ implies that $$ \textnormal{dist}(\partial I_{k},\partial I_{k}^*) \leq \textnormal{dist}(\partial  B(z_{k},R_{k}),\partial B(z_{k},CR_{k}))=(C-1)R_{k}.  $$ On the other hand, by observing that in every step above we can replace $I_{k}^{*}:=B(z_{k},CR_{k}),$ in view of the inclusion $$ I_k\subset B(z_{k},2R_{k}) \subset I_{k}^{*}:= B(z_{k},CR_{k}),$$ we have
      $$ (C-2)R_k=\textnormal{dist}(\partial I_{k}^{*},\partial B(z_{k},2R_{k}) ) \leq \textnormal{dist}(\partial I_{k},\partial I_{k}^*) . $$
  \end{itemize}Consequently,
 \begin{align*}
   \textnormal{diam}(I_k) &\asymp R_{k}\asymp \textnormal{dist}(\partial  B(z_{k},2R_{k}),\partial B(z_{k},CR_{k})) \\
   &\asymp \textnormal{dist}(\partial I_{k},\partial I_{k}^*).
 \end{align*} To show that $\textnormal{dist}(\partial I_{k}^*,\partial I_{k})  \leq|x|,$ observe that from Remark \ref{eGasump}, $e_{G}\in \cup_{j}I_j,$  and because of $x\in G\setminus \cup_{j}I_j,$  $$\textnormal{dist}(\partial I_{k}^*,\partial I_{k})\lesssim \textnormal{diam}(\cup_jI_j) \lesssim d(x,e_G)=|x|. $$
So, we have guaranteed the existence of a positive constant, which we again denote by $c>0,$ such that, $$\{x\in G: x\in G\setminus \cup_jI_j^*\}\subset\{x\in G:\textnormal{ for all } z\in {I_k},\,\,\, 4c|z|\leq |x| \}.$$
     So, by a suitable variable change of variables  and by using \eqref{ThLiWeak}, we have 
     \begin{align*}
        M_k:= \sup_{z\in I_k}\sum_{\ell=0}^\infty\int\limits_{ G\setminus  \cup_j I_j^*} & |\kappa_{\ell}(z^{-1}x)  - \kappa_{\ell}(x)|dx\\
         &=\sup_{z\in I_k}\sum_{\ell=0}^\infty\int\limits_{ G\setminus  \cup_j I_j^*}|2^{-\ell Q}\kappa_{\ell}(2^{-\ell}\cdot z^{-1}x)  - 2^{-\ell Q}\kappa_{\ell}(2^{-\ell }\cdot x)|dx\\
         &\leq \sup_{z\in I_k} \sum_{\ell=0}^\infty\int\limits_{ |x|>4c|z|}|2^{-\ell Q}\kappa_{\ell}(2^{-\ell}\cdot z^{-1}x)  - 2^{-\ell Q}\kappa_{\ell}(2^{-\ell }\cdot x)|dx\\
         &\leq \sum_{\ell=0}^\infty \sup_{z\in G}           \int\limits_{ |x|>4c|z|}|2^{-\ell Q}\kappa_{\ell}(2^{-\ell}\cdot z^{-1}x)  - 2^{-\ell Q}\kappa_{\ell}(2^{-\ell }\cdot x)|dx\\
          &=  \sum_{\ell=0}^\infty \mathscr{I}_{\ell}\lesssim \sum_{\ell=0}^\infty  2^{-\ell\varepsilon_0}=O(1).
     \end{align*}Because $$ M_k:= \sup_{z\in I_k}\sum_{\ell=0}^\infty\int\limits_{ G\setminus  \cup_j I_j^*}  |\kappa_{\ell}(z^{-1}x)  - \kappa_{\ell}(x)|dx\lesssim \sum_{\ell=0}^\infty  2^{-\ell\varepsilon_0},$$ with the right hand side of the inequality being independent of $k,$ we conclude that $M$ in \eqref{calderonkernels} is finite.

   According to Step 2, the vector-valued interpolation between the  \eqref{weak(1,1)inequality} and the boundedness of $W$ from $L^2(G)$ into $L^{2}(G,\ell^2(\mathbb{N}_0^n)),$ allows us to conclude \eqref{DCMR} for all $1<p\leq 2.$  
   
   Step 3. Let us assume that $2\leq p<\infty,$ and let $f\in L^p(G,\ell^2).$ So, for a.e. $x\in G,$ $f(x)$ is a sequence in $\ell^2.$ If $p'$ is  the conjugate exponent of $p,$ $1<p'\leq 2,$ by using that $\psi_{j}(\mathcal{R})$ is self-adjoint, we have 
   \begin{align*}
     & \Vert Wf \Vert_{L^p(G,\ell^2)}
      \\
      &=\sup_{\Vert h\Vert_{L^{p'}(G,\ell^2) }=1}\int\limits_{G}(Wf(x),h(x))_{\ell^2}dx=\sup_{\Vert h\Vert_{L^{p'}(G,\ell^2) }=1}\int\limits_{G}\sum_{j=0}^{\infty}\psi_{j}(\mathcal{R})f(x)h_j(x)dx\\
      &=\sup_{\Vert h\Vert_{L^{p'}(G,\ell^2) }=1}\int\limits_{G}\sum_{j=0}^{\infty}f(x)\psi_{j}(\mathcal{R})^{*}h_j(x)dx\\
       &=\sup_{\Vert h\Vert_{L^{p'}(G,\ell^2) }=1}\int\limits_{G}(f(x),Wh(x))_{\ell^2}dx,\,\,\,\,\,Wh(x):=\{\psi_\ell(\mathcal{R})h_\ell\}_{\ell=0}^\infty,\\
       &\leq \Vert f \Vert_{L^p(G,\ell^2)} \sup_{\Vert h\Vert_{L^{p'}(G,\ell^2) }=1}\Vert Wh  \Vert_{L^{p'}(G,\ell^2)}.
   \end{align*} By using Theorem \ref{Theoremr} with $r=2,$ for $\Vert h\Vert_{L^{p'}(G,\ell^2) }=1,$ we have that
   \begin{align*}
       \Vert Wh  \Vert_{L^{p'}(G,\ell^2)}=\left\Vert \left(\sum_{\ell=0}^{\infty} |\psi_{\ell}(\mathcal{R})h_\ell|^{2}    \right)^{\frac{1}{2}}\right\Vert_{L^{p'}(G)}\leq C_{p'}\left\Vert \left(\sum_{\ell=0}^{\infty} |h_\ell(x)|^{2}    \right)^{\frac{1}{2}}\right\Vert_{L^{p'}(G)}= C_{p'},
   \end{align*}where the constant $C_{p'}>0,$ came from Theorem \ref{Theoremr}. Consequently, we have proved  \eqref{DCMR} for all $2\leq p<\infty,$ in view of the boundedness of $W$ from $L^p(G)$ into $L^p(G,\ell^2)$ for all  $2\leq p<\infty.$ \\
   
   Step 4.
The proof of the left hand side of \eqref{LPTequ} for non-negative $f$ is as follows. Now, let us denote by $(E(\lambda))_{\lambda \geq 0}$  the spectral resolution associated to $\mathcal{R},$ and for every $\pi\in \widehat{G}$ denote by $(E_{\pi}(\lambda))_{\lambda \geq 0}$ the spectral resolution of $\pi(\mathcal{R}).$  We observe that by duality
\begin{align*}
\Vert f\Vert_{L^p(G)}&\lesssim \sup\{|\int_{G}f(x)g(x)dx|:g\in \mathcal{D}(G), \,g\geq 0,\,\Vert g\Vert_{L^{p'}}=1  \}\\
&=\sup\{|\int_{G}\sum_{l\in\mathbb{N}_0}[\psi_{l}(\mathcal{R})f](x)g(x)dx|:g\in \mathcal{D}(G), \,g\geq 0,\, \Vert g\Vert_{L^{p'}}=1  \}\\
&=\sup\{|\int_{G}\sum_{l\in\mathbb{N}}[E^{(l)}\psi_{l}(\mathcal{R})f](x)g(x)dx\\&\hspace{3cm}+\int_{G}[E^{(0)}\psi_{0}(\mathcal{R})f(x)]g(x)dx|:g\in \mathcal{D}(G),  \,g\geq 0,\, \Vert g\Vert_{L^{p'}}=1   \}\\
&=\sup\{|\int_{G}\sum_{l\in\mathbb{N}}[\psi_{l}(\mathcal{R})f](x)E^{(l)}g(x)dx\\&\hspace{3cm}+\int_{G}[\psi_{0}(\mathcal{R})f(x)]E^{(0)}g(x)dx|:g\in \mathcal{D}(G), \,g\geq 0,\,\Vert g\Vert_{L^{p'}}=1   \},
\end{align*}
where $E^{(l)}:=\psi_{\ell-1}(\mathcal{R})+\psi_{\ell}(\mathcal{R})+\psi_{\ell+1}(\mathcal{R}),$ for $l\geq 1,$ and  $E^{(0)}:=\psi_{0}(\mathcal{R})+\psi_{1}(\mathcal{R})+\psi_{2}(\mathcal{R}).$  Consequently, we have
\begin{align*}
&\Vert f\Vert_{L^p(G)} \\
&\leq  \sup\{  \int_{G}\left|\sum_{l\in\mathbb{N}_0}|[\psi_{l}(\mathcal{R})f](x)|^{2}\right|^{\frac{1}{2}}\left|\sum_{l\in\mathbb{N}_0}|E^{(l)}g(x)|^{2}\right|^{\frac{1}{2}}dx:g\in \mathcal{D}(G),\,g\geq 0,\, \Vert g\Vert_{L^{p'}}=1   \}\\
&\leq  \sup\left\{  \left\Vert\left[\sum_{l\in\mathbb{N}_0}|[\psi_{l}(\mathcal{R})f](x)|^{2} \right]^{\frac{1}{2}}\right\Vert_{L^p(G)}    \left\Vert \left[\sum_{l\in\mathbb{N}_0}|E^{(l)}g(x)|^{2}\right]^{\frac{1}{2}}\right\Vert_{L^{p'}(G)}:g\in \mathcal{D}(G) , \,g\geq 0,\, \Vert g\Vert_{L^{p'}}=1  \right\}.
\end{align*}
Because 
\begin{equation}
        \left\Vert \left[\sum_{l\in\mathbb{N}_0}|E^{(l)}g(x)|^{2}\right]^{\frac{1}{2}}\right\Vert_{L^{p'}(G)}\asymp  \left\Vert \left[\sum_{l\in\mathbb{N}_0}|\psi_{\ell}(\mathcal{R})g(x)|^{2}\right]^{\frac{1}{2}}\right\Vert_{L^{p'}(G)} \lesssim \Vert g\Vert_{L^{p'}}=1,
\end{equation}
we obtain
$$ \Vert f\Vert_{L^p(G)} \lesssim \left\Vert\left[\sum_{l\in\mathbb{N}_0}|[\psi_{l}(\mathcal{R})f](x)|^{2} \right]^{\frac{1}{2}}\right\Vert_{L^p(G)} $$ for all $1<p<\infty$.\\
Step 5. Let us assume that $f\in L^p(G)$ is real-valued. Decompose $f=f^+-f^-,$ as the difference of two non-negative functions, where $f^+,f^-\in L^p(G), $ and $|f|=f^++f^{-}.$ Because, $f^+,f^-\leq |f|,$ the Minkowski inequality implies
\begin{align*}
     & \left\Vert\left[\sum_{l\in\mathbb{N}_0}|[\psi_{l}(\mathcal{R})f](x)|^{2} \right]^{\frac{1}{2}}\right\Vert_{L^p(G)}\\
     &\leq \left\Vert\left[\sum_{l\in\mathbb{N}_0}|[\psi_{l}(\mathcal{R})f^+](x)|^{2} \right]^{\frac{1}{2}}\right\Vert_{L^p(G)}+\left\Vert\left[\sum_{l\in\mathbb{N}_0}|[\psi_{l}(\mathcal{R})f^-](x)|^{2} \right]^{\frac{1}{2}}\right\Vert_{L^p(G)}\\
     &\leq C_{p}( \Vert f^+\Vert_{L^p(G)}+ \Vert f^-\Vert_{L^p(G)} ) \leq 2C_{p}\Vert f\Vert_{L^p(G)}.
\end{align*} So, we have proved the right hand side of \eqref{LPTequ}. For the proof of the left hand side, we only need to repeat the proof made in Step 4. The proof for the weak (1,1) inequality is similar. Indeed, 
\begin{align*}
   & \left|\left\{x\in G:\left(\sum_{\ell=0}^\infty|\psi_\ell(\mathcal{R}) f(x)|^2   \right)^{\frac{1}{2}}>t \right\}\right|\\
    &\leq \left|\left\{x\in G:\left(\sum_{\ell=0}^\infty|\psi_\ell(\mathcal{R}) f^+(x)|^2   \right)^{\frac{1}{2}}>\frac{t}{2} \right\}\right|+\left|\left\{x\in G:\left(\sum_{\ell=0}^\infty|\psi_\ell(\mathcal{R}) f^-(x)|^2   \right)^{\frac{1}{2}}>\frac{t}{2} \right\}\right|\\
    &\leq \frac{2C}{t}\Vert f^+ \Vert_{L^1(G)}+ \frac{2C}{t}\Vert f^- \Vert_{L^1(G)}\\
    &\leq  \frac{4C}{t}\Vert f \Vert_{L^1(G)}.
\end{align*} A similar analysis can be used for the proof  of Step 6. So, the proof of the Littlewood-Paley Theorem is complete.
\end{proof}

We end this section with the proof of the vector-valued inequalities presented in  Theorem \ref{Theoremr}.

\begin{proof}[Proof of Theorem \ref{Theoremr}]
Define the vector-valued operator
\begin{equation}\label{EQprimerita'}
    W:L^2(G,\ell^2(\mathbb{N}_0))_f\rightarrow L^2(G,\ell^2(\mathbb{N}_0)),
\end{equation} by
\begin{equation}
    W(\{f_\ell\}_{\ell=0}^\infty)=\{\psi_\ell(\mathcal{R})f_\ell\}_{\ell=0}^\infty,
\end{equation} where $L^2(G,\ell^2(\mathbb{N}_0))_f$ is the set of sequences $\{f_\ell\}_{\ell=0}^\infty,$ with compact support in the $\ell$-variables.  We claim that $W:L^2(G,\ell^2(\mathbb{N}_0))\rightarrow L^2(G,\ell^2(\mathbb{N}_0))$ extends to  a bounded operator.
Indeed, let us observe that
\begin{equation}
  \Vert W(\{f_\ell\}_{\ell=0}^\infty)\Vert_{L^2(\ell^2)}^2=   \int_G\sum_\ell\vert \psi_\ell(\mathcal{R})f_\ell(x) \vert^2 dx\leq C^2 \sum_\ell\Vert f_\ell \Vert^2_{L^2(G)} = \Vert \{f_\ell\}_{\ell=0}^\infty\Vert_{L^2(\ell^2)}^2,
\end{equation}
where 
\begin{equation}
    C=\sup_{\ell\geq 0}\Vert \psi_\ell(\mathcal{R}) \Vert_{\mathscr{B}(L^2(G))}= \sup_{\ell\geq 0}\Vert \psi_\ell(\pi(\mathcal{R})) \Vert_{\mathscr{B}(L^2(\widehat{G}))}\leq \sup_{\ell\geq 0,\lambda\geq 0}|\psi(2^{-\ell}\lambda)|=O(1).
\end{equation}
Now, we will assume for a moment that for every $\ell\in \mathbb{N}_0,$ the family of operators $\{\psi_\ell(\mathcal{R})\}$ is uniformly bounded from $L^1(G)$ into $L^{1,\infty}(G),$ this is, every operator $\psi_\ell(\mathcal{R})$ is of weak $(1,1)$ type  with the $(L^1(G),L^{1,\infty}(G))$-operator norm bounded with a constant independent of $\ell,$

\begin{equation}\label{theomega}
     \Omega:=\sup_\ell \Vert \psi_\ell(\mathcal{R})  \Vert_{\mathscr{B}(L^1(G),L^{1,\infty}(G))}<\infty.
\end{equation}This assumption allows us to show that \begin{equation}\label{EQprimerita}
    W:L^{1}(G,\ell^1(\mathbb{N}_0))\rightarrow L^{1,\infty}(G,\ell^1(\mathbb{N}_0)),
\end{equation} extends to a bounded operator.
Indeed, if we define $e_{\ell'}(\ell)=\delta_{\ell,\ell'},$ observe that
\begin{align*}
    &\Vert W(\{f_\ell\}_{\ell=0}^\infty)  \Vert_{L^{1,\infty}(G,\ell^1(\mathbb{N}_0))} =\Vert \{\psi_\ell(\mathcal{R})f_\ell\}_{\ell=0}^\infty  \Vert_{L^{1,\infty}(G,\ell^1(\mathbb{N}_0))}\\
    &=\Vert \sum_{\ell'=0}^\infty\{e_{\ell'}(\ell)\psi_\ell(\mathcal{R})f_\ell\}_{\ell=0}^\infty  \Vert_{L^{1,\infty}(G,\ell^1(\mathbb{N}_0))}\\
      &\lesssim \sum_{\ell'=0}^\infty \Vert \{e_{\ell'}(\ell)\psi_\ell(\mathcal{R})f_\ell\}_{\ell=0}^\infty  \Vert_{L^{1,\infty}(G,\ell^1(\mathbb{N}_0))}.
\end{align*}   The fact that $\Vert \{e_{\ell'}(\ell)\psi_\ell(\mathcal{R})f_\ell\}_{\ell=0}^\infty  \Vert_{L^{1,\infty}(G,\ell^1(\mathbb{N}_0))}=\sup_{t>0}\,t\cdot |\{x\in G: \vert \psi_{\ell'}(\mathcal{R})f_{\ell'}(x)|   >t\}|,$ implies that 
\begin{align*}
    &\Vert W(\{f_\ell\}_{\ell=0}^\infty)  \Vert_{L^{1,\infty}(G,\ell^1(\mathbb{N}_0))} =\sum_{\ell'=0}^\infty\sup_{t>0}\,t\cdot |\{x\in G: \vert \psi_{\ell'}(\mathcal{R})f_{\ell'}(x)|   >t\}|\\
    &=\sum_{\ell'=0}^\infty \Vert \psi_{\ell'}(\mathcal{R})f_{\ell'}\|_{L^{1,\infty}(G)}   \\
    &\leq \Omega \sum_{\ell'=0}^\infty \Vert f_{\ell'}\|_{L^{1}(G)}=\int_{G}\sum_{\ell'=0}^\infty|\psi_{\ell'}(\mathcal{R})f_{\ell'}(x)|dx\\\
    &=\Vert \{f_\ell\}_{\ell=0}^\infty\Vert_{L^1(G,\ell^1(\mathbb{N}_0))}.
\end{align*}
Now, if $1<p<2,$ there exists $\theta\in (0,1)$ such that $\frac{1}{p}=\frac{\theta}{1}+\frac{1-\theta}{2}.$ In view of \ref{EQprimerita'} and \ref{EQprimerita},  by the  Lions-Peetre vector-valued interpolation theorem we have
\begin{equation}\label{TRE}
    W:L^p(G,\ell^p(\mathbb{N}_0))\rightarrow L^p(G,\ell^p(\mathbb{N}_0)),
\end{equation} extends to a bounded operator for all $1<p\leq 2$.
 Here, we have used that in this vector-valued context,\begin{equation}L^p(G,\ell^p(\mathbb{N}_0))=(L^{1,\infty}(G,\ell^1(\mathbb{N}_0)),L^2(G,\ell^2(\mathbb{N}_0)))_{\theta,p},\end{equation}with the usual notation of real interpolation (see Section \ref{independenceinterpolation} below).
Because $W$ is a symmetric operator on $L^2(G),$ the duality argument allows us to prove that \eqref{TRE} extends to a bounded operator for all $2\leq  p<\infty.$ So, the boundedness of $W$ for all $1<p<\infty$ is proved once that we have proved the estimate  \eqref{theomega}. For this, we observe that in view of the weak (1,1) estimate in the H\"ormander-Mihlin theorem (see Corollary 4.12 of \cite{FR}), we have
\begin{align*}
   &\Omega:= \sup_\ell \Vert \psi_\ell(\mathcal{R})  \Vert_{\mathscr{B}(L^1(G),L^{1,\infty}(G))}
   \lesssim  \sup_{\ell;\alpha\leq [Q/2]+1, t>0 }t^\alpha|\psi_\ell^{(\alpha)}(t)|\\
   &=\sup_{\ell;\alpha\leq [Q/2]+1, t>0 }t^\alpha2^{-\ell\alpha}|\psi^{(\alpha)}(2^{-\ell}t)|\\
    &\lesssim 1,
\end{align*}where in the last line we have used that the compactly supported function $\psi,$ satisfies estimate of the type
\begin{equation}
    |\psi^{(\alpha)}(\lambda)|\leq C_\alpha \lambda^{-\alpha},\,\,\lambda\neq 0,
\end{equation} and so
\begin{align*}
    t^\alpha2^{-\ell\alpha}|\psi^{(\alpha)}(2^{-\ell}t)|\leq C_\alpha t^\alpha2^{-\ell\alpha}(2^{-\ell}t)^{-\alpha}=C_\alpha.
\end{align*}
Now, we claim that 
\begin{equation}\label{thelpofW}
    W:L^p(G,\ell^r(\mathbb{N}_0))\rightarrow L^p(G,\ell^r(\mathbb{N}_0)),\,\,\,1<r<\infty.
\end{equation} To do so, we will prove that 
\begin{equation}\label{weakvectorvalued}
    W:L^{1}(G,\ell^r(\mathbb{N}_0))\rightarrow L^{1,\infty}(G,\ell^r(\mathbb{N}_0)),\,\,\,1<r<\infty.
\end{equation}
extends to a bounded operator together with a interpolation argument coming from \eqref{EQprimerita'}. For the proof of \eqref{weakvectorvalued}, we need to show that there exists a constant $C>0$ independent of $\{f_\ell\}\in L^1(G,\ell^r(\mathbb{N}_0))$ and $t>0,$ such that 
\begin{equation}
    \left|\left\{x\in G:\left(\sum_{\ell=0}^\infty|\psi_\ell(\mathcal{R}) f_\ell(x)|^r   \right)^{\frac{1}{r}}>t \right\}\right|\leq \frac{C}{t}\Vert \{f_\ell\} \Vert_{L^1(G,\ell^r(\mathbb{N}_0)}.
\end{equation} So, fix $\{f_\ell\}\in L^1(G,\ell^r(\mathbb{N}_0))$ and $t>0,$ and let $h(x):= \left(\sum_{\ell=0}^\infty| f_\ell(x)|^r   \right)^{\frac{1}{r}}, $ apply the Calder\'on-Zygmund decomposition Lemma to $h\in L^1(G),$ under the identification $G\simeq \mathbb{R}^n,$  (see e.g. Hebish \cite{Hebish}) in order to obtain a disjoint collection $\{I_j\}_{j=0}^{\infty}$ of disjoint open sets such that
\begin{itemize}
    \item $h(x)\leq t,$ for $a.e.$ $x\in G\setminus \cup_{j\geq 0}I_j,$\\
    
    \item $\sum_{j\geq 0}|I_j|\leq \frac{C}{t}\Vert h\Vert_{L^1(G)},$ and\\ 
    
    \item $t\leq \frac{1}{|I_j|}\int_{I_j}h(x)dx\leq 2t,$ for all $j.$
\end{itemize} Now, we will define a suitable decomposition of $f_\ell,$ for every $\ell\geq 0.$ Recall that   every $I_j$ is diffeomorphic to an open cube on $\mathbb{R}^n,$ that it is bounded, and that  $I_j\subset B(z_j,2R_j),$ where $z_j\in I_j$ (see Hebish \cite{Hebish}).  Let us define, for every $\ell,$ and $x\in I_j,$
\begin{equation}
    g_\ell(x):=\frac{1}{|I_j|}\int\limits_{I_j}f_\ell(y)dy,\,\,\,b_{\ell}(x)=f_\ell(x)-g_{\ell}(x).
\end{equation} and for $x\in G\setminus\cup_{j\geq 0}I_j,$
\begin{equation}
    g_\ell(x)=f_\ell(x),\,\,\, b_\ell(x)=0.
\end{equation} So, for a.e.  $x\in G,$ $f_\ell(x)=g_\ell(x)+b_\ell(x).$
Note that for every $1<r<\infty,$ $\Vert \{g_\ell\}\Vert_{L^r(\ell^r)}^r\leq t^{r-1}\Vert \{f_\ell\} \Vert_{L^1(\ell^r)}, $ indeed for $x\in I_j$, Minkowsky integral inequality gives,
\begin{align*}
    \left(\sum_{\ell=0}^\infty|g_\ell(x)|^r\right)^{\frac{1}{r}} &\leq  \left(\sum_{\ell=0}^\infty\left|   \frac{1}{|I_j|}\int\limits_{I_j}f_\ell(y)dy \right|^r\right)^{\frac{1}{r}}\leq \frac{1}{|I_j|}\int\limits_{I_j}\left(\sum_{\ell=0}^\infty|   f_\ell(y) |^r\right)^{\frac{1}{r}}dy\\
    &\leq 2t.
\end{align*}
Consequently, we have
\begin{align*}
    \Vert \{g_\ell\}\Vert_{L^r(\ell^r)}^r&=\int\limits_{G}\sum_{\ell=0}^\infty|g_\ell(x)|^rdx=\sum_{j}\int\limits_{I_j}\sum_{\ell=0}^\infty|g_\ell(x)|^rdx+\int\limits_{G\setminus \cup_{j}I_j}\sum_{\ell=0}^\infty|g_\ell(x)|^rdx\\
    &=\sum_{j}\int\limits_{I_j}\sum_{\ell=0}^\infty|g_\ell(x)|^rdx+\int\limits_{G\setminus \cup_{j}I_j}\sum_{\ell=0}^\infty|f_\ell(x)|^rdx\\
    &\leq \sum_{j}\int\limits_{I_j}(2t)^{r}dx+\int\limits_{G\setminus \cup_{j}I_j}h(x)^rdx\\
    &\lesssim t^{r}\sum_{j}|I_j|+\int\limits_{G\setminus \cup_{j}I_j}h(x)^{r-1}h(x)dx\\
    &\leq t^{r}\times \frac{C}{t}\Vert h\Vert_{L^1(G)}+t^{r-1}\int\limits_{G\setminus \cup_{j}I_j}h(x)dx\lesssim t^{r-1}\Vert h\Vert_{L^1(G)}\\
    &=t^{r-1}\Vert \{f_\ell\} \Vert_{L^1(\ell^r)}.
\end{align*}
Now, by using   the Minkowski and the Chebyshev inequality, we obtain 
\begin{align*}
    & \left|\left\{x\in G:\left(\sum_{\ell=0}^\infty|\psi_\ell(\mathcal{R}) f_\ell(x)|^r   \right)^{\frac{1}{r}}>t \right\}\right|\\
    &\leq  \left|\left\{x\in G:\left(\sum_{\ell=0}^\infty|\psi_\ell(\mathcal{R}) g_\ell(x)|^r   \right)^{\frac{1}{r}}>\frac{t}{2} \right\}\right|+ \left|\left\{x\in G:\left(\sum_{\ell=0}^\infty|\psi_\ell(\mathcal{R}) b_\ell(x)|^r   \right)^{\frac{1}{r}}>\frac{t}{2} \right\}\right|\\
    &\leq \frac{2^r}{t^r}\int_{G}\sum_{\ell=0}^\infty|\psi_\ell(\mathcal{R}) g_\ell(x)|^rdx+ \left|\left\{x\in G:\left(\sum_{\ell=0}^\infty|\psi_\ell(\mathcal{R}) b_\ell(x)|^r   \right)^{\frac{1}{r}}>\frac{t}{2} \right\}\right|.
\end{align*} In view of \eqref{TRE}, $W:L^r(G,\ell^r(\mathbb{N}_0))\rightarrow L^r(G,\ell^r(\mathbb{N}_0)),$ extends to a bounded operator and 
\begin{equation}
    \int_{G}\sum_{\ell=0}^\infty|\psi_\ell(\mathcal{R}) g_\ell(x)|^rdx=\Vert W\{g_\ell\}\Vert_{L^r(\ell^r)}^r\lesssim  \Vert \{g_\ell\}\Vert_{L^r(\ell^r)}^r\leq t^{r-1}\Vert \{f_\ell\}\Vert_{L^1(\ell^r)}.
\end{equation}
Consequently,
 \begin{align*}
    & \left|\left\{x\in G:\left(\sum_{\ell=0}^\infty|\psi_\ell(\mathcal{R}) f_\ell(x)|^r   \right)^{\frac{1}{r}}>t \right\}\right|\\
    &\leq  \frac{1}{t}\Vert \{f_\ell\}\Vert_{L^1(\ell^r)}+ \left|\left\{x\in G:\left(\sum_{\ell=0}^\infty|\psi_\ell(\mathcal{R}) b_\ell(x)|^r   \right)^{\frac{1}{r}}>\frac{t}{2} \right\}\right|.
\end{align*}
Now, we only need to prove that
\begin{equation}
     \left|\left\{x\in G:\left(\sum_{\ell=0}^\infty|\psi_\ell(\mathcal{R}) b_\ell(x)|^r   \right)^{\frac{1}{r}}>\frac{t}{2} \right\}\right|\lesssim  \frac{1}{t}\Vert \{f_\ell\}\Vert_{L^1(\ell^r)}.   
\end{equation} 

Taking into account that $b_\ell\equiv 0$ on $G\setminus \cup_j I_j,$ we have that
\begin{equation}
    b_\ell=\sum_{k}b_{\ell,k},\,\,\,b_{\ell,k}(x)=b_{\ell}(x)\cdot 1_{I_k}(x).
\end{equation} Let us assume that $I_{j}^*$ is a open set, such that $|I_{j}^*|=K|I_{j}|$ for some $K>0,$ and $\textnormal{dist}(\partial I_{j}^*,\partial I_{j})\geq 4c\textnormal{dist}(\partial I_{j},e_{G}),$ where $c$ is defined in \eqref{EQestimateweak(1,1)} and $e_{G}$ is the identity element of $G$.  So, by the Minkowski inequality we have,
\begin{align*}
     & \left|\left\{x\in G:\left(\sum_{\ell=0}^\infty|\psi_\ell(\mathcal{R}) b_{\ell}(x)|^r   \right)^{\frac{1}{r}}>\frac{t}{2} \right\}\right|\\
      &=\left|\left\{x\in \cup_j I_j^*:\left(\sum_{\ell=0}^\infty|\psi_\ell(\mathcal{R}) b_\ell(x)|^r   \right)^{\frac{1}{r}}>\frac{t}{2} \right\}\right|+\left|\left\{x\in G\setminus  \cup_j I_j^*:\left(\sum_{\ell=0}^\infty|\psi_\ell(\mathcal{R}) b_\ell(x)|^r   \right)^{\frac{1}{r}}>\frac{t}{2} \right\}\right|\\
      &\leq \left|\left\{x\in G:x\in \cup_j I_j^* \right\}\right|+\left|\left\{x\in G\setminus  \cup_j I_j^*:\left(\sum_{\ell=0}^\infty|\psi_\ell(\mathcal{R}) b_\ell(x)|^r   \right)^{\frac{1}{r}}>\frac{t}{2} \right\}\right|.
      \end{align*} Since $$  \left|\left\{x\in G:x\in \cup_j I_j^* \right\}\right| \leq \sum_{j}|I_j^*|,  $$ we have
      \begin{align*}  & \left|\left\{x\in G:\left(\sum_{\ell=0}^\infty|\psi_\ell(\mathcal{R}) b_\ell(x)|^2   \right)^{\frac{1}{2}}>\frac{t}{2} \right\}\right|\\
      &\leq\sum_{j}|I_j^*|+\left|\left\{x\in G\setminus  \cup_j I_j^*:\left(\sum_{\ell=0}^\infty|\psi_\ell(\mathcal{R}) b_\ell(x)|^2   \right)^{\frac{1}{2}}>\frac{t}{2} \right\}\right|\\
      &=K\sum_{j}|I_j|+\left|\left\{x\in G\setminus  \cup_j I_j^*:\left(\sum_{\ell=0}^\infty|\psi_\ell(\mathcal{R}) b_\ell(x)|^2   \right)^{\frac{1}{2}}>\frac{t}{2} \right\}\right|\\
      &\leq \frac{CK}{t}\Vert f\Vert_{L^1(G,\ell^r)}+\left|\left\{x\in G\setminus  \cup_j I_j^*:\left(\sum_{\ell=0}^\infty|\psi_\ell(\mathcal{R}) b_\ell(x)|^2   \right)^{\frac{1}{2}}>\frac{t}{2} \right\}\right|.
  \end{align*} Observe that  the Chebyshev inequality implies
  \begin{align*}
     & \left|\left\{x\in G\setminus  \cup_j I_j^*:\left(\sum_{\ell=0}^\infty|\psi_\ell(\mathcal{R}) b_\ell(x)|^r  \right)^{\frac{1}{r}}>\frac{t}{2} \right\}\right|\\
     &\leq\frac{2}{t}\int\limits_{ G\setminus  \cup_j I_j^*} \left(\sum_{\ell=0}^\infty|\psi_\ell(\mathcal{R}) b_\ell(x)|^r   \right)^{\frac{1}{r}}dx\\
     &=\frac{2}{t}\int\limits_{ G\setminus  \cup_j I_j^*} \left(\sum_{\ell=0}^\infty\left|\left(\psi_\ell(\mathcal{R})\left(\sum_{k} b_{\ell,k}\right)   \right)(x)\right|^r\right)^{\frac{1}{r}}dx\\
     &=\frac{2}{t}\int\limits_{ G\setminus  \cup_j I_j^*} \Vert\{(\psi_\ell(\mathcal{R})(\sum_{k} b_{\ell,k}) (x)\}_{\ell=0}^\infty\Vert_{\ell^r(\mathbb{N}_0)} dx\\
     &=\frac{2}{t}\int\limits_{ G\setminus  \cup_j I_j^*} \Vert\{\sum_{k} (\psi_\ell(\mathcal{R})b_{\ell,k})(x) \}_{\ell=0}^\infty\Vert_{\ell^r(\mathbb{N}_0)} dx\\
     &\leq \frac{2}{t}\sum_{k}\int\limits_{ G\setminus  \cup_j I_j^*} \left(\sum_{\ell=0}^\infty\left|\left(\psi_\ell(\mathcal{R})b_{\ell,k}   \right)(x)\right|^r\right)^{\frac{1}{r}}dx.
  \end{align*}
Now, if $\kappa_\ell$ is the right convolution Calder\'on-Zygmund kernel of     $\psi_\ell(\mathcal{R}) ,$ (see Remark \ref{theremarkof}), and by using that $\int_{I_k}b_{k,\ell}(y)dy=0,$ we have that
\begin{align*}
    \left(\sum_{\ell=0}^\infty\left|\left(\psi_\ell(\mathcal{R})b_{\ell,k}   \right)(x)\right|^r\right)^{\frac{1}{r}}&=\left(\sum_{\ell=0}^\infty\left|b_{\ell,k}\ast \kappa_{\ell}(x)   \right|^r\right)^{\frac{1}{r}}\\
    &=\left(\sum_{\ell=0}^\infty\left| \int\limits_{I_k}\kappa_\ell(y^{-1}x)b_{\ell,k}(y)dy-\kappa_{\ell}(x)\int\limits_{I_k}b_{\ell,k}(y)dy \right|^r\right)^{\frac{1}{r}}\\
       &=\left(\sum_{\ell=0}^\infty\left| \int\limits_{I_k}(\kappa_\ell(y^{-1}x)-\kappa_{\ell}(x))b_{\ell,k}(y)dy \right|^r\right)^{\frac{1}{r}}.
\end{align*} Now, we will proceed as follows. By using that
$
    |b_{\ell,k}(y)|^r\leq \sum_{\ell'=0}^{\infty}|b_{\ell',k}(y)|^r,
$ we have, by an application of the Minkowsky integral inequality,
\begin{align*}
    &\left(\sum_{\ell=0}^\infty\left|\left(\psi_\ell(\mathcal{R})b_{\ell,k}   \right)(x)\right|^r\right)^{\frac{1}{r}}= \left(\sum_{\ell=0}^\infty\left| \int\limits_{I_k}(\kappa_\ell(y^{-1}x)-\kappa_{\ell}(x))b_{\ell,k}(y)dy \right|^r\right)^{\frac{1}{r}}\\
    &\leq \int\limits_{I_k}\left(   \sum_{\ell=0}^\infty |\kappa_\ell(y^{-1}x)-\kappa_{\ell}(x)|^r|b_{\ell,k}(y)|^r\right)^{\frac{1}{r}}dy\\
    &\leq  \int\limits_{I_k}\left(   \sum_{\ell'=0}^{\infty}|b_{\ell',k}(y)|^r\right)^{\frac{1}{r}}\left(   \sum_{\ell=0}^\infty |\kappa_\ell(xy^{-1})-\kappa_{\ell}(x)|^r\right)^{\frac{1}{r}}dy.
\end{align*} 
Consequently, we deduce,
\begin{align*}
   & \frac{2}{t}\sum_{k}\int\limits_{ G\setminus  \cup_j I_j^*} \left(\sum_{\ell=0}^\infty\left|\left(\psi_\ell(\mathcal{R})b_{\ell,k}   \right)(x)\right|^r\right)^{\frac{1}{r}}dx\\
   &\leq \frac{2}{t} \sum_{k}\int\limits_{ G\setminus  \cup_j I_j^*} \int\limits_{I_k}\left(   \sum_{\ell'=0}^{\infty}|b_{\ell',k}(y)|^r\right)^{\frac{1}{r}}\left(   \sum_{\ell=0}^\infty |\kappa_\ell(y^{-1}x)-\kappa_{\ell}(x)|^r\right)^{\frac{1}{r}}dy                   dx\\
   &= \frac{2}{t} \sum_{k} \int\limits_{I_k}  \int\limits_{ G\setminus  \cup_j I_j^*} \left(   \sum_{\ell'=0}^{\infty}|b_{\ell',k}(y)|^r\right)^{\frac{1}{r}}\left(   \sum_{\ell=0}^\infty |\kappa_\ell(y^{-1}x)-\kappa_{\ell}(x)|^r\right)^{\frac{1}{r}}dx                  dy\\
    &= \frac{2}{t} \sum_{k} \int\limits_{I_k}  \left(   \sum_{\ell'=0}^{\infty}|b_{\ell',k}(y)|^r\right)^{\frac{1}{r}}  \int\limits_{ G\setminus  \cup_j I_j^*} \left(   \sum_{\ell=0}^\infty |\kappa_\ell(y^{-1}x)-\kappa_{\ell}(x)|^r\right)^{\frac{1}{r}}dxdy.
\end{align*}
Because $\textnormal{dist}(\partial I_{j}^*,\partial I_{j})\geq 4c\textnormal{dist}(\partial I_{j},e_{G}),$ for $x\in G\setminus \cup_jI_j^*,$   for $y\in I_{k},$ the analysis in \eqref{Estimateforboundaries}, shows that  $4c|y|=4c\times \textnormal{dist}(y,e_G) \lesssim  \textnormal{dist}(\partial I_{k}^*,\partial I_{k})  \leq|x|.$ So,  $$\{x\in G: x\in G\setminus \cup_jI_j^*\}\subset\{x\in G:\textnormal{ for all } z\in {I_k},\,\,\, 4c|z|\leq |x| \}.$$ Now, from Remark \ref{theremarkof}, the identity \eqref{EQestimateweak(1,1)}, and the estimate \eqref{ThLiWeak}, we deduce \begin{align*}
  & \int\limits_{ G\setminus  \cup_j I_j^*} \left(   \sum_{\ell=0}^\infty |\kappa_\ell(y^{-1}x)-\kappa_{\ell}(x)|^r\right)^{\frac{1}{r}}dx \leq \int\limits_{ G\setminus  \cup_j I_j^*} \sum_{\ell=0}^\infty |\kappa_\ell(y^{-1}x)-\kappa_{\ell}(x)|dx\\
   &\leq  \sum_{\ell=0}^\infty \int\limits_{ G\setminus  \cup_j I_j^*}|\kappa_\ell(y^{-1}x)-\kappa_{\ell}(x)|dx\\\
   &\leq   \sum_{\ell=0}^\infty \int\limits_{|x| >4c|y|}|2^{-\ell Q}\kappa_\ell(2^{-\ell }\cdot y^{-1}x)-2^{-\ell Q}\kappa_{\ell}(2^{-\ell }\cdot x)|dx\\
   &\lesssim  \sum_{\ell=0} 2^{-\ell \varepsilon_0}=O(1).
\end{align*}Thus, we have proved that
\begin{align*}
      &\left|\left\{x\in G:\left(\sum_{\ell=0}^\infty|\psi_\ell(\mathcal{R}) b_\ell(x)|^r   \right)^{\frac{1}{r}}>\frac{t}{2} \right\}\right|\lesssim  \frac{2}{t} \sum_{k} \int\limits_{I_k}  \left(   \sum_{\ell'=0}^{\infty}|b_{\ell',k}(y)|^r\right)^{\frac{1}{r}}dy \\
      &=\frac{2}{t}  \int\limits_{\cup_{k}I_{k}}  \left(   \sum_{\ell'=0}^{\infty}|b_{\ell'}(y)|^r\right)^{\frac{1}{r}}dy \\
      &\lesssim \frac{1}{t}\Vert \{f_\ell\}\Vert_{L^1(\ell^r)}. 
\end{align*} This, the proof of the weak (1,1) inequality is complete and we have that \begin{equation}\label{weakvectorvalued2}
    W:L^{1}(G,\ell^r(\mathbb{N}_0))\rightarrow L^{1,\infty}(G,\ell^r(\mathbb{N}_0)),\,\,\,1<r<\infty,
\end{equation} extends to a bounded operator. As an application of the vector-valued Lions-Peetre interpolation theorem between \eqref{EQprimerita'} and \eqref{weakvectorvalued2} we obtain that $W$ in \eqref{thelpofW} extends to a bounded operator and together with the duality argument we complete the proof.
\end{proof}

\section{Homogeneous and inhomogeneous Besov spaces}\label{Besov}

Let $\mathcal{R}$ be a (left-invariant) positive Rockland operator on a graded Lie group $G.$ In order to define the family of Besov spaces on $G,$ let us assume that $\mathcal{R}$ is homogeneous of degree $\nu>0$ and let us fix a dyadic
decomposition of its spectrum: we choose a function $\psi\in C^{\infty}_{0}(\mathbb{R})$  supported
in  $[1/4, 2],$ $\psi=1$ on $[1/2,1].$ Denote by $\psi_{l}$ the function $\psi_{l}(t)=\psi(2^{-l}t),$ $t\in \mathbb{R}.$ For some smooth  compactly supported function $\psi_{0}$ we have
\begin{eqnarray}\label{deco1}
\sum_{l\in\mathbb{N}_{0}}\psi_{l}(\lambda)=1,\,\,\, \text{for every}\,\,\, \lambda>0.
\end{eqnarray}
With notations above we define (left) Besov spaces associated to a (left-invariant) positive Rockland operator as follows.
\begin{definition}\label{nonhomogeneaous}
Let $r\in \mathbb{R},$ $0<p<\infty$ and $0< q\leq \infty.$ The homogeneous Besov space $\dot{B}^{r}_{p,q,\psi,\mathcal{R}}(G)$ associated to $(\mathcal{R},(\psi_l)_{l})$ consists of those $f\in \mathcal{D}'(G)$ satisfying
\begin{equation}
\Vert f\Vert_{\dot{B}^{r}_{p,q,\psi,\mathcal{R}}(G)}:=\left(   \sum_{l\in\mathbb{N}_{0}}2^{  {\frac{l}{\nu}} rq}\Vert \psi_{l}(\mathcal{R})f\Vert^{q}_{L^{p}{(G)}}\right)^{\frac{1}{q}}<\infty,
\end{equation}
for $0<q<\infty,$ and for $q=\infty,$
\begin{equation}
\Vert f\Vert_{\dot{B}^{r}_{p,\infty,\psi,\mathcal{R}}(G)}:=\sup_{l\in\mathbb{N}_{0}}2^{ {\frac{l}{\nu}}  r}\Vert \psi_{l}(\mathcal{R})f\Vert_{L^{p}{(G)}}<\infty. 
\end{equation}
Analogously, the inhomogeneous Besov space ${B}^{r}_{p,q,\psi,\mathcal{R}}(G)$ is defined as the space of distributions $f\in \mathcal{D}'(G)$ satisfying 
\begin{equation}
\Vert f\Vert_{{B}^{r}_{p,q,\psi,\mathcal{R}}(G)}:=\left(   \sum_{l\in\mathbb{N}_{0}}2^{  \frac{l}{\nu}  rq}\Vert \psi_{l}(I+\mathcal{R})f\Vert^{q}_{L^{p}{(G)}}\right)^{\frac{1}{q}}<\infty,
\end{equation}
if $0<q<\infty$ and, for $q=\infty,$
\begin{equation}
\Vert f\Vert_{{B}^{r}_{p,\infty,\psi,\mathcal{R}}(G)}:=\sup_{l\in\mathbb{N}_{0}}2^{  \frac{l}{\nu}  r}\Vert \psi_{l}(I+\mathcal{R})f\Vert_{L^{p}{(G)}}<\infty. 
\end{equation}
\end{definition}
Homogeneous and inhomogeneous Besov spaces
do not depend on a particular choice of a positive Rockland operator $\mathcal{R}$ and of the sequence of smooth functions $\psi_{l}.$ We will prove this fact in the following section (see Theorem \ref{independence}). Now, we prove the following embedding properties of Besov spaces. We use the simplified notation motivated by  Theorem \ref{independence}, $$(\dot{B}^{r}_{p,q}(G),\Vert\cdot \Vert_{\dot{B}^{r}_{p,q}(G)} )=(\dot{B}^{r}_{p,q,\psi,\mathcal{R}}(G),\Vert\cdot \Vert_{\dot{B}^{r}_{p,q,\psi,\mathcal{R} }(G)} )$$ and $$({B}^{r}_{p,q}(G),\Vert\cdot \Vert_{{B}^{r}_{p,q}(G)} )=({B}^{r}_{p,q,\psi,\mathcal{R}}(G),\Vert\cdot \Vert_{{B}^{r}_{p,q,\psi,\mathcal{R} }(G)} ).$$
For Sobolev spaces $H^{r,p}(G)$ and $\dot{H}^{r,p}(G)$ and their properties we refer to \cite[Section 4]{FR2}.

We also note that similar results would hold if we chose right-invariant (instead of left-invariant) Rockland operator in the definition of Besov spaces, see Remark \ref{REM:right}.

\begin{theorem}\label{Emb1}
Let $G$ be a graded Lie group of homogeneous dimension $Q$ and let $r\in \mathbb{R}.$ Then 
\begin{itemize}
\item[(1)] $\dot{B}^{r+\varepsilon}_{p,q_1}(G)\hookrightarrow \dot{B}^{r}_{p,q_1}(G)\hookrightarrow \dot{B}^{r}_{p,q_2}(G)\hookrightarrow \dot{B}^{r}_{p,\infty}(G),$  $\varepsilon>0,$ $0<p\leq \infty,$ $0<q_{1}\leq q_2\leq \infty.$\\
\item[(2)]  $\dot{B}^{r+\varepsilon}_{p,q_1}(G)\hookrightarrow \dot{B}^{r}_{p,q_2}(G)$, $\varepsilon>0,$ $0<p\leq \infty,$ $1\leq q_2<q_1<\infty.$\\
\item[(3)]  $\dot{B}^{r_1}_{p_1,q}(G)\hookrightarrow \dot{B}^{r_2}_{p_2,q}(G),$  $1\leq p_1\leq p_2\leq \infty,$ $0<q<\infty,$ $r_{1}\in\mathbb{R}$ and $r_2=r_1- Q(\frac{1}{p_1}-\frac{1}{p_2}).$\\
\item[(4)] $\dot{H}^{r}(G)=\dot{B}^{r}_{2,2}(G)$ and $\dot{B}^{r}_{p,p}(G)\hookrightarrow \dot{H}^{r,p}(G)\hookrightarrow \dot{B}^{r}_{p,2}(G),$ $1<p\leq 2.$\\
\item[(5)] $\dot{B}^{r}_{p,1}(G)\hookrightarrow L^{q}(G), $ $1\leq p\leq q\leq \infty,$ $r= Q(\frac{1}{p}-\frac{1}{q}).$
\end{itemize}
\end{theorem}
\begin{proof}
For the proof of $(1)$ we observe that
\begin{align*}
\Vert f \Vert_{\dot{B}^{r}_{p,\infty}}&=\sup_{s\in\mathbb{N}_0}2^{r \frac{s}{\nu}  }\Vert \psi_{s}(\mathcal{R}) f\Vert_{L^{p}}\leq \Vert \{ 2^{ \frac{s}{\nu}  r}\Vert \psi_{s}(\mathcal{R})f\Vert_{L^{p}} \}_{s\in\mathbb{N}_0} \Vert_{l^{q_2}(\mathbb{N}_0)}\equiv\Vert f \Vert_{\dot{B}^{r}_{p,q_2}}\\
&\leq \Vert \{ 2^{ \frac{s}{\nu} r}\Vert \psi_{s}(\mathcal{R})f\Vert_{L^{p}} \}_{s\in\mathbb{N}_0} \Vert_{l^{q_1}(\mathbb{N}_0)}\equiv \Vert f \Vert_{\dot{B}^{r}_{p,q_1}}\\
&\leq \Vert \{ 2^{ \frac{s}{\nu}(r+\varepsilon)} \Vert \psi_{s}(\mathcal{R}^{\frac{1}{\nu}})f\Vert_{L^{p}} \}_{s\in\mathbb{N}_0} \Vert_{l^{q_1}(\mathbb{N}_0)}\\&\equiv\Vert f \Vert_{\dot{B}^{r+\varepsilon}_{p,q_1}}.\\
\end{align*}
For the proof of $(2)$ we use H\"older inequality as follows,
\begin{align*}
\Vert f \Vert_{\dot{B}^{r}_{p,q_2}}&=\Vert \{ 2^{ \frac{s}{\nu} r}\Vert \psi_{s}(\mathcal{R})f\Vert_{L^{p}} \}_{s\in\mathbb{N}_0} \Vert_{l^{q_2}(\mathbb{N}_0)}\\
&=\Vert \{ 2^{  \frac{s}{\nu} (r+\varepsilon)-\frac{s}{\nu}\varepsilon}\Vert \psi_{s}(\mathcal{R})f\Vert_{L^{p}} \}_{s\in\mathbb{N}_0} \Vert_{l^{q_2}(\mathbb{N}_0)}\\
&\leq \Vert \{ 2^{   \frac{s}{\nu}(r+\varepsilon)q_2}\Vert \psi_{s}(\mathcal{R})f\Vert^{q_2}_{L^{p}} \}_{s\in\mathbb{N}_0} \Vert^{\frac{1}{q_2}}_{l^{\frac{q_1}{q_2}}(\mathbb{N}_0)}[\sum_{s\in\mathbb{N}_0}2^{ -\frac{ \frac{s}{\nu}\varepsilon q_2q_1 }{q_1-q_2  }    }]^{\frac{1}{q_1}-\frac{1}{q_2}}\\
&\lesssim \Vert f \Vert_{\dot{B}^{r+\varepsilon}_{p,q_1}}.
\end{align*}
In order to prove $(3)$ we use Nikolskii inequality from Theorem \ref{Nikolskii}. So, by  taking into account the estimate \begin{equation}\Vert \psi_{s}(\mathcal{R})f\Vert_{L^{p_2}}\leq C2^{s \frac{Q}{\nu}(\frac{1}{p_1}-\frac{1}{p_2})}\Vert \psi_{s}(\mathcal{R})f\Vert_{L^{p_1}},\end{equation} 
we deduce
$$ \left(   \sum_{s\in\mathbb{N}_{0}}2^{   \frac{s}{\nu}r_2q}\Vert \psi_{s}(\mathcal{R})f\Vert^{q}_{L^{p_2}{(G)}}\right)^{\frac{1}{q}}\lesssim \left(   \sum_{s\in\mathbb{N}_{0}}2^{  \frac{s}{\nu}[r_{2}+ Q(\frac{1}{p_1}-\frac{1}{p_2}) ]q}\Vert \psi_{s}(\mathcal{R})f\Vert^{q}_{L^{p_1}{(G)}}\right)^{\frac{1}{q}} .$$
Now we will prove $(4),$ that is $\dot{B}^{r}_{p,p}(G)\hookrightarrow \dot{H}^{r,p}(G)\hookrightarrow \dot{B}^{r}_{p,2}(G),$ for $1<p\leq 2.$ In fact, we have
\begin{align*}
\Vert f \Vert_{\dot{H}^{r,p}}^{p}&\equiv\Vert \mathcal{R}^{\frac{r}{\nu}}f\Vert^{p}_{L^{p}} =\Vert  \int\limits_{0}^{\infty}\lambda^{\frac{r}{\nu}}dE(\lambda)f   \Vert_{L^p}^{p}\\
&=\Vert  \sum_{s\in \mathbb{Z}}\int_{2^s}^{2^{s+1}} \lambda^{ \frac{r}{\nu} }dE(\lambda)f   \Vert_{L^p}^{p}\leq   \sum_{s\in \mathbb{Z}} \Vert \int_{2^s}^{2^{s+1}} \lambda^{\frac{r}{\nu}  }dE(\lambda)f   \Vert_{L^p}^{p}\\
&=   \sum_{s\in \mathbb{Z}}2^{\frac{s}{\nu}rp} \Vert \int_{2^s}^{  2^{s+1}  }2^{-\frac{s}{\nu}r} \lambda^{\frac{r}{\nu}}dE(\lambda)f   \Vert_{L^p}^{p}\\ &\asymp  \sum_{s\in \mathbb{Z}}2^{\frac{s}{\nu}rp} \Vert \int_{2^s}^{2^{s+1}}dE(\lambda)f   \Vert_{L^p}^{p}\asymp  \sum_{s\in \mathbb{N}_{0}}2^{\frac{s}{\nu}rp} \Vert \int_{2^s}^{ 2^{s+1} }\psi_{s}(\lambda)dE(\lambda)f   \Vert_{L^p}^{p}\\&= \sum_{s\in \mathbb{N}_0}2^{\frac{s}{\nu}rp} \Vert \psi_{s}(\mathcal{R})f   \Vert_{L^p}^{p}=\Vert f\Vert^p_{\dot{B}^{r}_{p,p}}.
\end{align*}
For the other embedding we use the following version of the Minkowski  integral inequality 
$$ (\sum_{j}(\int\limits_{X} f_j(x) d\mu(x))^{\alpha})^{\frac{1}{\alpha}}\leq\int\limits_{X}(\sum_{j}f_j^{\alpha}(x))^{\frac{1}{\alpha}}d\mu(x) ,\,\,\,f_j(x)\geq 0,\,\,a.e. \,x\in X,$$
where $(X,\mu)$ is a $\sigma$-finite measure space, and $\alpha=\frac{2}{p}.$ So, we get
\begin{align*}
\Vert f \Vert_{\dot{B}^r_{p,2}}&=\left( \sum_{s\in \mathbb{N}_0} 2^{2r\frac{s}{\nu}}\Vert\psi_{s}(\mathcal{R})f \Vert^{2}_{L^p}\right)^{\frac{1}{2}}\\&=\left( \sum_{s\in \mathbb{N}_0} 2^{2r\frac{s}{\nu}}[\int_{G}\vert\psi_{s}(\mathcal{R})f(x) \vert^{p}dx]^{\frac{2}{p}}\right)^{\frac{p}{2} \cdot \frac{1}{p}}\leq\left[ \int_{G} [  \sum_{s\in\mathbb{N}_0}2^{2\frac{s}{\nu}r}  |\psi_{s}(\mathcal{R})f(x)       |^{\frac{2}{p}p}dx]^{\frac{p}{2}}   \right]^{\frac{1}{p}}\\&=\left[ \int_{G} [  \sum_{s\in\mathbb{N}_0}2^{2\frac{s}{\nu}r}  |\psi_{s}(\mathcal{R})f(x)       |^{2}dx]^{\frac{p}{2}}   \right]^{\frac{1}{p}}\\
&=\Vert   [\sum_{s\in\mathbb{N}_0}2^{2\frac{s}{\nu}r}  |\psi_{s}(\mathcal{R})f(x)       |^{2}dx]^{\frac{1}{2}}\Vert_{L^{p}}\asymp \Vert   [\sum_{s\in\mathbb{N}_0}  |\psi_{s}(\mathcal{R}^{\frac{r}{\nu}})f(x)       |^{2}dx]^{\frac{1}{2}}\Vert_{L^{p}}\\
&\asymp \Vert \mathcal{R}^{\frac{r}{\nu}}f \Vert_{L^p}=\Vert f\Vert_{\dot{H}^{r,p}},
\end{align*}
using Littlewood-Paley theorem (Theorem \ref{LPT}). We observe that in the embedding $\dot{B}^{r}_{p,p}(G)\hookrightarrow \dot{H}^{r,p}(G)\hookrightarrow \dot{B}^{r}_{p,2}(G),$ if   $p=2$ then $\dot{H}^{r,2}(G)=\dot{B}^r_{2,2}(G).$ Now, for the proof of $(5)$ we use Nikolskii inequality,
\begin{align*}
\Vert f \Vert_{L^{q}}&=\Vert \int_{\widehat{G}}\textrm{Tr}[\pi(x)\widehat{f}(\pi)]d\pi   \Vert_{L^q}\\
&=\Vert\sum_{s\in\mathbb{N}_0}\int_{\widehat{G}}\textrm{Tr}[\pi(x)\psi_{s}[\pi(\mathcal{R})]\widehat{f}(\pi)]d\pi   \Vert_{L^q}\\
&\leq \sum_{s\in\mathbb{N}_0} \Vert\int_{\widehat{G}}\textrm{Tr}[\pi(x)\psi_{s}[\pi(\mathcal{R})]\widehat{f}(\pi)]d\pi   \Vert_{L^q}\\
&= \sum_{s\in\mathbb{N}_0} \Vert \psi_{s}(\mathcal{R})f\Vert_{L^q}\leq  \sum_{s\in\mathbb{N}_0} 2^{\frac{Q}{\nu}(\frac{1}{p}-\frac{1}{q })}\Vert \psi_{s}(\mathcal{R})f\Vert_{L^p}\\
&=\Vert f\Vert_{\dot{B}^{{Q}(\frac{1}{p}-\frac{1}{q})}_{p,1}}.
\end{align*}
This completes the proof.
\end{proof}
In the following theorem we present embeddings properties for inhomogeneous Besov spaces ${B}^{r}_{p,q}(G).$  The proof is similar to the homogeneous case, so we omit it.
\begin{theorem}\label{Emb2}
Let $G$ be a graded Lie group of homogeneous dimension $Q$ and let $r\in\mathbb{R}.$ Then 
\begin{itemize}
\item[(1)] $B^{r+\varepsilon}_{p,q_1}(G)\hookrightarrow B^{r}_{p,q_1}(G)\hookrightarrow B^{r}_{p,q_2}(G)\hookrightarrow B^{r}_{p,\infty}(G),$  $\varepsilon>0,$ $0<p\leq \infty,$ $0<q_{1}\leq q_2\leq \infty.$\\
\item[(2)]  $B^{r+\varepsilon}_{p,q_1}(G)\hookrightarrow B^{r}_{p,q_2}(G)$, $\varepsilon>0,$ $0<p\leq \infty,$ $1\leq q_2<q_1<\infty.$\\
\item[(3)]  $B^{r_1}_{p_1,q}(G)\hookrightarrow B^{r_2}_{p_2,q}(G),$  $1\leq p_1\leq p_2\leq \infty,$ $0<q<\infty,$ $r_{1}\in\mathbb{R}$ and $r_2=r_1- Q(\frac{1}{p_1}-\frac{1}{p_2}).$\\
\item[(4)] $H^{r}(G)=B^{r}_{2,2}(G)$ and $B^{r}_{p,p}(G)\hookrightarrow H^{r,p}(G)\hookrightarrow B^{r}_{p,2}(G),$ $1<p\leq 2.$\\
\item[(5)] $B^{r}_{p,1}(G)\hookrightarrow L^{q}(G), $ $1\leq p\leq q\leq \infty,$ $r= Q(\frac{1}{p}-\frac{1}{q}).$
\end{itemize}
\end{theorem}

\begin{remark}[Right Besov spaces]
\label{REM:right}
Throughout this section we have considered Besov spaces associated to (left-invariant) positive Rockland operators. A similar formulation of homogeneous and inhomogeneous (right) Besov spaces can be obtained if we choose (right-invariant) positive Rockland operators.  It can be shown that these spaces satisfy (right) versions of Theorems \ref{Emb1} and \ref{Emb2}. When properties that we want to consider hold for left and right Besov spaces, we omit the prefixes left and right, nevertheless, we consider in the proofs the case of (left) Besov spaces.
\end{remark}

\section{Independence of Rockland operators and interpolation  properties}\label{independenceinterpolation} 

In this section we prove the independence of the choice of Rockland operator and the dyadic partition $\psi_{l}$ in the definition of Besov spaces. For this, we show that Besov spaces can be obtained as interpolation of Sobolev spaces. If $X_0$ and $X_1$ are Banach spaces, the main notion in real interpolation theory is the $K$-functional, defined by
\begin{equation}
K(f,t)=\inf\{\Vert f_{0}\Vert_{X_0}+t\Vert f_1 \Vert_{X_1}:f=f_0+f_1, f_0\in X_0,f_1\in X_1\},\,\,t\geq 0.
\end{equation}
If $0<\theta<1$ and $1\leq q<\infty,$ the real interpolation space $X_{\theta,q}:=(X_0,X_1)_{\theta,q}$ is defined by those vectors $f\in X_0+X_1$ satisfying 
\begin{equation}
\Vert f \Vert_{\theta,q}=\left(  \int_{0}^{\infty}  ( t^{-\theta}K(f,t))^{q}\frac{dt}{t}   \right)^{\frac{1}{q}}<\infty\,\, \text{if} \,\,q<\infty,
\end{equation}
and for $q=\infty$
\begin{equation}
\Vert f \Vert_{\theta,q} =\sup_{t>0}  t^{-\theta}K(f,t)<\infty.
\end{equation}
For our purposes, the following discrete form (see \cite{handbook}, p. 1136) will be useful
\begin{equation}
\Vert f \Vert_{\theta,q}\asymp \inf\left\{\left( \sum_{k\in\mathbb{Z}}\max\{\Vert f_k \Vert_{X_0}, 2^{k}\Vert f_{k}\Vert_{X_1}\}^{q}\right)^{\frac{1}{q}}:\,\,\,f=\sum_{k\in\mathbb{Z}}{2^{k\theta}}f_k\right\} .
\end{equation}
with $1\leq q<\infty.$

\begin{theorem}\label{independence}
Let $G$ be a graded Lie group, and let $\mathcal{R}$ and $\mathcal{R}'$ be two positive Rockland operators with homogeneity degrees $\nu>0$ and $\nu'>0,$ respectively. If $(\psi_{l})_{l}$ and $(\psi'_{l})_l$ are sequences satisfying \eqref{deco1}, $1<p<\infty$ and $1\leq q<\infty,$  the spaces $\dot{B}^{r}_{p,q,\psi,\mathcal{R}}(G)$ and $\dot{B}^{r}_{p,q,\psi',\mathcal{R'}}(G)$ coincide and  have equivalent norms, as well as the spaces  ${B}^{r}_{p,q,\psi,\mathcal{R}}(G)$ and ${B}^{r}_{p,q,\psi',\mathcal{R'}}(G).$ We also have the following interpolation properties:
\begin{equation}\label{eq11} 
B^{r}_{p,q}(G)=(H^{b,p}(G),H^{a,p}(G))_{\theta,q},\quad a<r<b,\quad r=b(1-\theta)+a\theta,
\end{equation}
and 
\begin{equation}\label{eq2} \dot{B}^{r}_{p,q}(G)=(\dot{H}^{b,p}(G),\dot{H}^{b,a}(G))_{\theta,q},\quad a<r<b,\quad r=b(1-\theta)+a\theta.
\end{equation}
\end{theorem}
\begin{proof}
It was proved in \cite[Theorem 4.4.20]{FR2}, that the definition of (homogeneous and inhomogeneous) Sobolev spaces ($\dot{H}^{r,p}(G)$ and $H^{r,p}(G)$, respectively) does not depend on the choice of Rockland operators. Hence the independence of the choice of Rockland operators and of the dyadic decomposition $\psi_{l}$ in the case of Besov spaces  would follow if we show that Besov spaces can be obtained by interpolation of Sobolev spaces. So, it suffices to prove \eqref{eq11} and \eqref{eq2}. First we will show that for $r>0,$ $\dot{B}^{r}_{p,q}=(\dot{H}^{r_1,p},\dot{H}^{r_0,p})_{\theta,q}$ where $0<r_{0}<r<r_{1}=r_{0}+\nu,$ $r=r_{1}+(r_0-r_1)\theta,$ and later we will deduce the general case from this fact. For $f\in \dot{H}^{r_{1},p}+\dot{H}^{r_0,p}$ we write
\begin{equation}
f=\sum_{l\geq 0}\psi_{l}(\mathcal{R})f=\sum_{l\geq 0}2^{l\theta}f_{l}',\quad f_{l}'=2^{-l\theta} \psi_{l}(\mathcal{R})f.
\end{equation}
Hence
\begin{align*}
\Vert f\Vert^{q}_{\theta,q}\lesssim \sum_{l\geq 0}\max\{\Vert f_{l}' \Vert_{\dot{H}^{r_1,p}},2^{l}\Vert f_{l}' \Vert_{\dot{H}^{r_0,p}}\}^{q}. 
\end{align*}
Now, if $(E_{\lambda})_{\lambda\geq 0}$ denotes the spectral resolution associated to $\mathcal{R},$ we have

\begin{align*}
\Vert f\Vert^{q}_{\theta,q} &\lesssim \sum_{l\geq 0}2^{-\theta lq}\max\{\Vert \psi_{l}(\mathcal{R}) f\Vert_{\dot{H}^{r_1,p}},2^{l}\Vert \psi_{l}(\mathcal{R}) f\Vert_{\dot{H}^{r_0,p}}\}^{q}\\
&=\sum_{l\geq 0} 2^{-\theta lq}\max\{\Vert \int_{2^{l-1}}^{2^{l+1}}\psi_{l}(\lambda)dE_{\lambda} f\Vert_{\dot{H}^{r_1,p}},2^{l}\Vert \int_{2^{l-1}}^{2^{l+1}}\psi_{l}(\lambda)dE_{\lambda} f \Vert_{\dot{H}^{r_0,p}}\}^{q}\\
&=\sum_{l\geq 0} 2^{-\theta lq}\max\{\Vert \int_{2^{l-1}}^{2^{l+1}}\lambda^{\frac{r_1}{\nu}}\psi_{l}(\lambda)dE_{\lambda} f\Vert_{L^{p}},2^{l}\Vert \int_{2^{l-1}}^{2^{l+1}}\lambda^{\frac{r_0}{\nu}}\psi_{l}(\lambda)dE_{\lambda} f \Vert_{L^{p}}\}^{q}\\
&\asymp\sum_{l\geq 0} 2^{-\theta lq}\max\{\Vert 2^{ \frac{{r_1}l}{\nu}  } \int_{2^{l-1}}^{2^{l+1}}\psi_{l}(\lambda)dE_{\lambda} f\Vert_{L^{p}},2^{l}\Vert 2^{  \frac{r_0 l}{\nu}  } \int_{2^{l-1}}^{2^{l+1}}\psi_{l}(\lambda)dE_{\lambda} f \Vert_{L^{p}}\}^{q}\\
&=\sum_{l\geq 0} 2^{-\theta lq}\max\{\Vert 2^{ \frac{{r_1}l}{\nu}  } \psi_{l}(\mathcal{R}) f\Vert_{L^{p}},2^{l}\Vert 2^{ \frac{{r_0}l}{\nu}  } \psi_{l}(\mathcal{R}) f \Vert_{L^{p}}\}^{q}.\\
\end{align*}
Since
$$ \max\{\Vert 2^{ \frac{{r_1}l}{\nu}  } \psi_{l}(\mathcal{R}) f\Vert_{L^{p}},2^{l}\Vert 2^{ \frac{{r_0}l}{\nu}  } \psi_{l}(\mathcal{R}) f \Vert_{L^{p}}\}=\max\{ 2^{ \frac{{r_1}l}{\nu}  },2^{ \frac{{r_0}l}{\nu} +l } \}\Vert\psi_{l}(\mathcal{R}) f \Vert_{L^{p}} ,$$
we obtain
\begin{align*}
\Vert f\Vert^{q}_{\theta,q} &\lesssim \sum_{l\geq 0} 2^{-\theta lq}\max\{ 2^{ \frac{{r_1}lq}{\nu}  },2^{ \frac{{r_0}lq}{\nu} +lq } \}\Vert\psi_{l}(\mathcal{R}) f \Vert_{L^{p}} ^{q}\\
&= \sum_{l\geq 0} \max\{ 2^{ \frac{{r_1}lq}{\nu} -\theta lq },2^{ \frac{{r_0}lq}{\nu} +lq -\theta lq} \}\Vert\psi_{l}(\mathcal{R}) f \Vert_{L^{p}} ^{q}\\
&= \sum_{l\geq 0} \max\{ 2^{-ql\theta( \frac{r_0-r_1}{\nu}+1)},2^{ql(1-\theta)(\frac{r_0-r_1}{\nu}+1)} \}2^{rql}\Vert\psi_{l}(\mathcal{R}^{\frac{1}{\nu}}) f \Vert_{L^{p}} ^{q}.\\
\end{align*}
Taking into account that $r_{0}-r_1+\nu=0$ we have
\begin{equation}
\Vert f\Vert^q_{\theta,q}\lesssim \sum_{l\geq 0} 2^{  \frac{rql}{\nu}  }\Vert\psi_{l}(\mathcal{R}) f \Vert_{L^{p}} ^{q} \equiv\Vert f \Vert^{q}_{\dot{B}^{r}_{p,q}}.
\end{equation}
Now, in order to proof the converse inequality we use the following estimate on the operator norm of $\psi_{l}$ for $l$ large enough :  $\Vert \psi_{l}(\mathcal{R})\Vert_{\mathcal{L}(L^p)}=O(1) ,$ which can be obtained by interpolation between the trivial estimate for $p=2,$ \eqref{theomega}, and the duality argument. We observe that by the Liouville theorem (see Geller \cite{geller_83} or \cite[Section 3.2.8]{FR2}), $\lambda=0$ is not an eigenvalue of $\mathcal{R}.$ So, we have,
\begin{align*}
\Vert f \Vert^q_{\dot{B}^r_{p,q}}&=\sum_{l\geq 0}2^{  \frac{lrq}{\nu} }\Vert \psi_{l}(\mathcal{R})f \Vert_{L^p}^q\\
&=\sum_{l\geq 0}2^{  \frac{lrq}{\nu} }\Vert \psi_{l}(\mathcal{R})\mathcal{R}^{-\frac{r_1}{\nu}}\mathcal{R}^{\frac{r_1}{\nu}} f \Vert_{L^p}^q\\
&=\sum_{l\geq 0}2^{  \frac{lrq}{\nu} }\Vert \int_{2^{l-1}}^{2^{l+1}}\psi_{l}(\lambda)\lambda^{-\frac{r_1}{\nu}}dE_{\lambda}\mathcal{R}^{\frac{r_1}{\nu}} f \Vert_{L^p}^q\\
&\asymp \sum_{l\geq 0}2^{  \frac{lrq}{\nu}-\frac{lr_1q}{\nu} }\Vert \int_{2^{l-1}}^{2^{l+1}}\psi_{l}(\lambda)dE_{\lambda}\mathcal{R}^{\frac{r_1}{\nu}} f \Vert_{L^p}^q\\
&= \sum_{l\geq 0}2^{  \frac{lrq}{\nu}-\frac{lr_1q}{\nu} }\Vert \psi_{l}(\mathcal{R})\mathcal{R}^{\frac{r_1}{\nu}} f \Vert_{L^p}^q\\
&\lesssim  \sum_{l\geq 0}2^{  \frac{lrq}{\nu}-\frac{lr_1q}{\nu}}\Vert \mathcal{R}^{\frac{r_1}{\nu}} f \Vert_{L^p}^q\lesssim C_{s}\Vert \mathcal{R}^{\frac{r_1}{\nu}} f \Vert_{L^p}^q.
\end{align*}
Hence
$$ \Vert f \Vert_{\dot{B}^r_{p,q}}\lesssim C_s\Vert f\Vert_{\dot{H}^{r_1,p}} . $$
In a similar way, we can prove the estimate
$$ \Vert f \Vert_{\dot{B}^r_{p,q}}\lesssim \Vert f\Vert_{\dot{H}^{r_0,p}}.  $$
So, we have the embedding $\dot{H}^{r_i,p}\hookrightarrow \dot{B}^{r}_{p,q} $ for $i=0,1.$ Hence $(\dot{H}^{r_1,p},\dot{H}^{r_0,p})_{\theta,q} \hookrightarrow \dot{B}^{r}_{p,q}.$ So we conclude that $\Vert f\Vert_{\dot{B}^{r}_{p,q}}\lesssim \Vert f\Vert_{\theta,q}.$ In the case where $r<0$ we observe that $(I+\mathcal{R})^{\frac{|r|-r}{\nu}}:\dot{B}^{|r|}_{p,q}\rightarrow \dot{B}^{r}_{p,q}$ is an isomorphism and for $0<r_0<|r|<r_{0}+\nu=r_1$ we obtain,
\begin{align*} \dot{B}^{r}_{p,q}=(I+\mathcal{R})^{\frac{|r|-r}{\nu}}(\dot{B}^{|r|}_{p,q})&=((I+\mathcal{R}))^{\frac{|r|-r}{\nu}}\dot{H}^{r_0,p},(I+\mathcal{R}))^{\frac{|r|-r}{\nu}}\dot{H}^{r_1,p})_{\theta,q}\\
&=(\dot{H}^{r_0+r-|r|,p},\dot{H}^{r_1+r-|r|,p})_{\theta,q},
\end{align*}
with $|r|=r_1+\theta(r_0-r_1).$ The general case where $a<r<b$ and $r=b(1-\theta)+a\theta$ now follows if we consider $r_{0}=r-\frac{\nu}{2},$ $r_{1}=r+\frac{\nu}{2}$ and by observing that
$$ r=\frac{1}{2}r_1+\frac{1}{2}r_2=r_1+\frac{1}{2}(r_2-r_1).$$
So we get
\begin{equation}
(\dot{H}^{b,p},\dot{H}^{a,p})_{\theta,q}=(\dot{H}^{r_1,p},\dot{H}^{r_0,p})_{\frac{1}{2},q}.
\end{equation}
Since $(\dot{H}^{r_1,p},\dot{H}^{r_0,p})_{\frac{1}{2},q}=\dot{B}^{r}_{p,q}$ we conclude the proof of the homogeneous case. An analogous proof can be adapted to the inhomogeneous case.
\end{proof}

\subsection{Interpolation inequalities in Besov spaces}
\label{sec:int}

In this subsection we consider the problem of interpolation inequalities on Besov spaces on graded Lie groups. The following theorem generalises a version for Besov spaces  in $\mathbb{R}^n$ proved by Machihara and Ozawa \cite{MaOza-2003}. In turn, this extended many other known families of inequalities, we refer to \cite{MaOza-2003} for the review of the literature.

\begin{theorem}
Let $\lambda,\mu,p,q,r$ and $\theta$ be real numbers. If $1\leq p,q\leq r\leq \infty,$ 
\begin{equation}\label{lambdamu}
Q(\frac{1}{p}-\frac{1}{r})<\lambda\quad \textnormal{\em and}\quad\mu<Q(\frac{1}{q}-\frac{1}{r}),
\end{equation}
then we have the following inequalities:
\begin{itemize}
\item[(i)] $\Vert f \Vert_{\dot{B}^{0}_{r,1}(G)}\leq C\Vert f \Vert_{\dot{B}^{\lambda}_{p,\infty}(G)}^{\theta}\Vert f \Vert_{\dot{B}^{\mu}_{q,\infty}(G)}^{1-\theta},\,\, \,\, f\in \dot{B}^{\lambda}_{p,\infty}(G)\cap \dot{B}^{\mu}_{q,\infty}(G),$\\

\item[(ii)] $\Vert f \Vert_{{B}^{0}_{r,1}(G)}\leq C \Vert f \Vert_{{B}^{\lambda}_{p,\infty}(G)}^{\theta}\Vert f \Vert_{{B}^{\mu}_{q,\infty}(G)}^{1-\theta},\,\, \,\, f\in {B}^{\lambda}_{p,\infty}(G)\cap {B}^{\mu}_{q,\infty}(G), $\\

\item[(iii)] $\Vert f\Vert_{L^{r}(G)}\,\leq\, C \Vert f \Vert_{{H}^{\lambda,p}(G)}^{\theta}\Vert f \Vert_{{H}^{\mu,q}(G)}^{1-\theta},\,\, \,\, f\in {H}^{\lambda,p}(G)\cap {H}^{\mu,q}(G), $
\end{itemize}
where 
\begin{equation} \theta(\lambda-Q(\frac{1}{p}-\frac{1}{r}))+(1-\theta)(\mu-Q(\frac{1}{q}-\frac{1}{r}))=0.
\end{equation}
\end{theorem}
\begin{proof}
In order to prove (i), we consider $f\in \dot{B}^{\lambda}_{p,\infty}(G)\cap \dot{B}^{\mu}_{q,\infty}(G)$ such that $f\not\equiv 0$ almost everywhere, with $\lambda,\mu,p$ and $q$ satisfying  \eqref{lambdamu}. If $\psi_{-1}$ is some smooth function supported in $[-1,\frac{1}{2}],$ then by using the fact that the corresponding Littlewood-Paley decomposition satisfies
\begin{equation}
\textnormal{supp}(\psi_{k})\cap \textnormal{supp}(\psi_{j})=\emptyset,\,\,|j-k|\geq 2,\,\, j,k\geq -1,
\end{equation}
we can write 
\begin{align*}
\Vert f\Vert_{\dot{B}^{0}_{r,1}(G)}&=\sum_{l=0}^{\infty}\Vert \psi_{l}(\mathcal{R})f\Vert_{L^{r}(G)}=\sum_{l=0}^{\infty}\Vert\sum_{k=0}^{\infty}\psi_{k}(\mathcal{R}) \psi_{l}(\mathcal{R})f\Vert_{L^{r}(G)}\\
&\leq \sum_{l=0}^{\infty}\sum_{k=-1}^{\infty}\Vert \psi_{k}(\mathcal{R})\psi_{l}(\mathcal{R})f\Vert_{L^{r}(G)}=\sum_{l=0}^{\infty}\sum_{k=l-1}^{l+1}\Vert \psi_{k}(\mathcal{R})\psi_{l}(\mathcal{R})f\Vert_{L^{r}(G)}\\
&=\sum_{l=0}^{\infty}\sum_{k=l-1}^{l+1}\Vert \psi_{l}(\mathcal{R})f\ast \mathscr{F}_{G}^{-1}[\psi_k(\pi[\mathcal{R}])]\Vert_{L^{r}(G)}.
\end{align*}
If we use the Young inequality, for $\frac{1}{r}+1=\frac{1}{m}+\frac{1}{s},$ we have
\begin{equation}\label{psi00}
\Vert \psi_{l}(\mathcal{R})f\ast \mathscr{F}_{G}^{-1}[\psi_k(\pi[\mathcal{R}])]\Vert_{L^{r}(G)}\leq  \Vert\mathscr{F}_{G}[\psi_k(\pi[\mathcal{R}])] \Vert_{L^m(G)}\Vert \psi_{l}(\mathcal{R})f \Vert_{L^{s}(G)}.
\end{equation}
By the action of the dilations $D_t$ on $G$ we have for $k\geq 1,$ and $r=2^{-\frac{k}{\nu}}$, that
\begin{align*}
 \Vert \mathscr{F}_{G}^{-1}[\psi_k(\pi[\mathcal{R}])] \Vert_{L^m(G)}^m&=\int_{G}| \mathscr{F}_{G}^{-1}(\psi_k(\pi(\mathcal{R})))(x)|^m\,dx
 \\ &=\int_{G}|\mathscr{F}^{-1}_{G}[\psi_0(2^{-k}\pi(\mathcal{R}))](x)|^m\,dx\\
&=\int_{G}| \mathscr{F}_{G}^{-1}[\psi_0(r^{\nu}\pi[\mathcal{R}])](x)|^m\,dx\\&=\int_{G}r^{-Qm}| \mathscr{F}_{G}^{-1}[\psi_0(\pi[\mathcal{R}])](D_{r^{-1}}x)|^m\,dx\\
&=\int_{G}r^{-Qm+Q}| \mathscr{F}_{G}^{-1}[\psi_0(\pi[\mathcal{R}])](x)|^m\,dx.\\
\end{align*}
Hence we obtain
$$\Vert  \mathscr{F}_{G}^{-1}[\psi_k(\pi[\mathcal{R}])] \Vert_{L^m(G)}=2^{\frac{kQ}{\nu}(1-\frac{1}{m})}\Vert  \mathscr{F}_{G}^{-1}[\psi_0(\pi[\mathcal{R}])]\Vert_{L^m(G)}\lesssim 2^{\frac{kQ}{\nu}(1-\frac{1}{m})}. $$
If $l\geq 2$, then by using \eqref{psi00} for $s=p,q$ simultaneously, we obtain
\begin{align*}
\Vert f\Vert_{\dot{B}^{0}_{r,1}(G)}
%&\leq \sum_{l=0}^{\infty}\sum_{k=l-1}^{l+1}\Vert \psi_{l}(\mathcal{R})f\ast \mathcal{F}^{-1}[\psi_k(\pi[\mathcal{R}])]\Vert_{L^{r}(G)}\\
&\leq \sum_{j=0}^{\infty}\sum_{k=j-1}^{j+1}\Vert \psi_{j}(\mathcal{R})f\ast  \mathscr{F}_{G}^{-1}[\psi_k(\pi[\mathcal{R}])]\Vert_{L^{r}(G)}\\
&\lesssim \sum_{j\geq l}\sum_{k=j-1}^{j+1}2^{\frac{kQ}{\nu}(\frac{1}{p}-\frac{1}{r})}\Vert \psi_{j}(\mathcal{R})f\Vert_{L^p(G)}+\sum_{j< l}\sum_{k=j-1}^{j+1}2^{\frac{kQ}{\nu}(\frac{1}{q}-\frac{1}{r})}\Vert \psi_{j}(\mathcal{R})f\Vert_{L^q(G)}\\
&\lesssim \sum_{j\geq l}2^{\frac{jQ}{\nu}(\frac{1}{p}-\frac{1}{r})}\Vert \psi_{j}(\mathcal{R})f\Vert_{L^p(G)}+\sum_{j< l}2^{\frac{jQ}{\nu}(\frac{1}{q}-\frac{1}{r})}\Vert \psi_{j}(\mathcal{R})f\Vert_{L^q(G)}\\
&\lesssim \sum_{j\geq l}2^{\frac{jQ}{\nu}(\frac{1}{p}-\frac{1}{r}-\frac{\lambda}{Q})}2^{\frac{j\lambda}{\nu}}\Vert \psi_{j}(\mathcal{R})f\Vert_{L^p(G)}+\sum_{j< l}2^{\frac{jQ}{\nu}(\frac{1}{q}-\frac{1}{r}-\frac{\mu}{Q})}2^{\frac{j\mu}{\nu}}\Vert \psi_{j}(\mathcal{R})f\Vert_{L^q(G)}\\
&\lesssim \sum_{j\geq l}2^{\frac{jQ}{\nu}(\frac{1}{p}-\frac{1}{r}-\frac{\lambda}{Q})} \Vert f \Vert_{\dot{B}^\lambda_{p,\infty}}+\sum_{j< l}2^{\frac{jQ}{\nu}(\frac{1}{q}-\frac{1}{r}-\frac{\mu}{Q})}\Vert f \Vert_{\dot{B}^\mu_{q,\infty}}.\\
&\lesssim 2^{\frac{lQ}{\nu}(\frac{1}{p}-\frac{1}{r}-\frac{\lambda}{Q})} \Vert f \Vert_{\dot{B}^\lambda_{p,\infty}}+2^{\frac{lQ}{\nu}(\frac{1}{q}-\frac{1}{r}-\frac{\mu}{Q})}\Vert f \Vert_{\dot{B}^\mu_{q,\infty}},
\end{align*}
the last inequality due to \eqref{lambdamu}.
If we put $A:=\Vert f \Vert_{\dot{B}^\lambda_{p,\infty}}\Vert f \Vert^{-1}_{\dot{B}^\mu_{q,\infty}},$ then
\begin{align*}2^{\frac{lQ}{\nu}(\frac{1}{p}-\frac{1}{r}-\frac{\lambda}{Q})} \Vert f \Vert_{\dot{B}^\lambda_{p,\infty}} &+2^{\frac{lQ}{\nu}(\frac{1}{q}-\frac{1}{r}-\frac{\mu}{Q})}\Vert f \Vert_{\dot{B}^\mu_{q,\infty}} \\
&=(2^{\frac{lQ}{\nu}(\frac{1}{p}-\frac{1}{r}-\frac{\lambda}{Q})}A^{1-\theta}+2^{\frac{lQ}{\nu}(\frac{1}{q}-\frac{1}{r}-\frac{\mu}{Q})}A^{-\theta}) \Vert f \Vert^{\theta}_{\dot{B}^\lambda_{p,\infty}}\Vert f \Vert^{1-\theta}_{\dot{B}^\mu_{q,\infty}}  .
\end{align*}
Let us define the positive parameter $\sigma=(\lambda-\frac{Q}{p}+\frac{Q}{r})-(\mu-\frac{Q}{q}+\frac{Q}{r}),$ and assume that $l$ satisfies
\begin{equation}\label{lll}
2^{\frac{l}{\nu}}\leq A^{\frac{1}{\sigma}}< 2^{\frac{l+1}{\nu}}.
\end{equation}
We can assume \eqref{lll} if we take $l=[\frac{\nu}{\sigma}\log_2(A)]$, where $[\,\cdot\,]$ denotes the integer part function on the real numbers. Then we have
\begin{align*}
2^{\frac{lQ}{\nu}(\frac{1}{p}-\frac{1}{r}-\frac{\lambda}{Q})}A^{1-\theta}+ &2^{\frac{lQ}{\nu}(\frac{1}{q}-\frac{1}{r}-\frac{\mu}{Q})}A^{-\theta}\\
&\lesssim A^{\frac{1}{\sigma}(\frac{Q}{p}-\frac{Q}{r}-\lambda)+\frac{1}{\sigma}(\lambda-\frac{Q}{p}+\frac{Q}{r})}+A^{\frac{1}{\sigma}(\frac{Q}{q}-\frac{Q}{r}-\mu)+\frac{1}{\sigma}(\mu-\frac{Q}{q}+\frac{Q}{r})}=2,
\end{align*}
where in the last estimate we have used  that
\begin{equation}
1-\theta=\frac{1}{\sigma}(\lambda-\frac{Q}{p}+\frac{Q}{r})
\;\textrm{ and }\; -\theta=\frac{1}{\sigma}(\mu-\frac{Q}{q}+\frac{Q}{r}).
\end{equation}
Hence we obtain
$$  \Vert f\Vert_{\dot{B}^{0}_{r,1}(G)}\lesssim  \Vert f \Vert_{\dot{B}^{\lambda}_{p,\infty}(G)}^{\theta}\Vert f \Vert_{\dot{B}^{\mu}_{q,\infty}(G)}^{1-\theta},$$
which shows the estimate (i). The inequality in (ii) can be proved in a similar way. Finally, we have (iii) if we use (i) together with the embeddings $\dot{B}^{0}_{r,1}(G)\hookrightarrow L^r(G)$ and $H^{\rho,r}(G)\hookrightarrow B^{\rho}_{r,\infty}(G)$ proved in Theorem \ref{Emb1}.
\end{proof}

\section{Localisation of Besov spaces on graded Lie groups}
\label{SEC:local}

In this section we prove local embedding properties of  Besov spaces $B^{r}_{p,q}(G)$ with the ones defined in a local way on  $\mathbb{R}^n$. First we recall the notion of Besov spaces on $\mathbb{R}^n.$ For $x,h\in\mathbb{R}^n$ and $f\in L^{p}(\mathbb{R}^n),$ let us denote 
\begin{eqnarray}
\Delta_{h}^{m}f(x):=\sum_{k=0}^m C^{k}_{m}(-1)^{m-k}f(x+kh)
\end{eqnarray}
and
\begin{equation}\label{other}
\omega^{m}_{p}(t,f):=\sup_{|h|\leq t}\Vert \Delta_{h}^{m}f \Vert_{L^{p}(\mathbb{R}^n)}.
\end{equation}
Then, by following \cite{bwang} for $r>0$ and  $1\leq p, q\leq\infty$, the Euclidean Besov space $B^{r}_{p,q}(\mathbb{R}^{n})$ can be considered endowed with the norm
\begin{equation}
\Vert f\Vert_{ B^{r}_{p,q}(\mathbb{R}^{n}) }=\Vert f\Vert_{L^{p}(\mathbb{R}^n)}+\sum_{m=0}^{n}\left(\int_{0}^{\infty}(t^{-r}\omega^{m}_{p}(t,f))^{q}dt\right)^{\frac{1}{q}}
\end{equation}
for $q<\infty$, and with an obvious modification in the case $q=\infty.$ By considering the property $(I-\mathcal{L})^{\frac{\alpha}{2}}(B^{r}_{p,q}(\mathbb{R}^n))=B^{r-\alpha}_{p,q}(\mathbb{R}^n),$ where $\mathcal{L}$ is the Laplace operator on $\mathbb{R}^n,$ for $r<0,$ we can consider on $B^{r}_{p,q}(\mathbb{R}^n),$ $1\leq p,q\leq \infty,$ the norm
\begin{equation}
\Vert f \Vert_{B^{r}_{p,q}(\mathbb{R}^n)}=\Vert (I-\mathcal{L})^{-\frac{s}{2}}f \Vert_{B^{s+r}_{p,q}(\mathbb{R}^n)},
\end{equation}
where $s$ is a fixed real satisfying $s+r>0.$ It is a known fact that the definitions of Besov spaces on $\mathbb{R}^n$ by using the functional \eqref{other} are equivalent to those using Littlewood-Paley partitions for the Laplacian on $\mathbb{R}^n,$ in a analogous way as we have defined Besov space on graded Lie groups by using Rockland operators. It can be obtained if in particular in  Definition \ref{nonhomogeneaous} we put $G=\mathbb{R}^n$ and $\mathcal{R}=\Delta_{x},$ the positive Laplacian over $\mathbb{R}^n.$   If we denote for a graded Lie group $G$ the localisation space by
\begin{equation}
B^{r}_{p,q}(G,loc)=\{f\in \mathcal{D}'(G):\phi\cdot f\in B^{r}_{p,q}(G),\textnormal{ for all }\phi\in C^{\infty}_{0}(G)\}
\end{equation}
we have the following result.
\begin{proposition}
If $B^{r}_{p,q}(G,loc)$ denotes the local Besov space defined above, then for all $r\in\mathbb{R},$ $1<p<\infty$ and $0<q\leq \infty$ we have 
\begin{equation}
B^{\frac{r}{\nu_1}}_{p,q}(G,loc)\subset B^{r}_{p,q}(\mathbb{R}^n,loc)\subset B^{\frac{r}{\nu_n}}_{p,q}(G,loc),
\end{equation}
where $\nu_1$ and $\nu_n$  are respectively the smallest and the largest weights of the
dilations.
\end{proposition}
\begin{proof}
It was proved in \cite[Theorem 4.4.24]{FR2} that the following embedding of local Sobolev spaces holds:
\begin{equation}
H^{\frac{s}{\nu_1},p}(G,loc)\subset H^{s,p}(\mathbb{R}^n,loc)\subset H^{\frac{s}{\nu_n},p}(G,loc), 
\end{equation}
for all $s\in\mathbb{R}.$ Thus, the result now follows by using real interpolation in the sense of Theorem \ref{independence}. 
\end{proof}

\section{Fourier multipliers and spectral multipliers}\label{Multipliers}

In this section we give results for the boundedness of spectral and of Fourier multipliers in Besov spaces on graded Lie groups.

\subsection{Negative results for left-invariant operators}

There are some restrictions on indices for Besov spaces on which left-invariant operators may be bounded.

\begin{theorem}
Let $G$ be a graded Lie group and let $T$ be a linear left-invariant
operator bounded from $B^{r}_{p,q}(G)$ (respectively, $\dot{B}^{r}_{p,q}(G)$) into $B^{\tilde{r}}_{\tilde{p},\tilde{q}}(G),$ (respectively, $\dot{B}^{\tilde{r}}_{\tilde{p},\tilde{q}}(G)),$ for $1\leq p,\tilde{p}<\infty,$  $ -\infty<r,\tilde{r}<\infty,$ and $0<q,\tilde{q}\leq \infty.$ If $1\leq \tilde{p}<p<\infty,$ then $T = 0.$ 
\end{theorem}
\begin{proof}
Let $|\cdot|$ be a homogeneous norm on $G.$ It is known  (see \cite[Lemma 3.2.5]{FR}) that
$$  \lim_{|h|\rightarrow \infty}\Vert f+\tau_hf \Vert_{L^{p}(G)}=2^{\frac{1}{p}}\Vert f\Vert_{L^{p}(G)},$$ where $\tau_{h}$ is defined by $\tau_{h}f(x)=f(hx)$, $x,h\in G.$ First, we will  prove the case where  $0<\tilde{q}<\infty$ in the inhomogeneous case. By the boundedness of $T$ we have $\Vert Tf\Vert_{B^{\tilde{r}}_{\tilde{p},\tilde{q}}(G)}\leq \Vert T\Vert \Vert f\Vert_{B^{{r}}_{{p},{q}}(G)},$ where $\Vert T \Vert=\Vert T\Vert_{B(r,p,q;\tilde{r},\tilde{p},\tilde{q})}$ is the usual operator norm.
So, for every $h\in G$ we have
$$\Vert T(f+\tau_hf)\Vert_{B^{\tilde{r}}_{\tilde{p},\tilde{q}}(G)}\leq C\Vert f+\tau_hf\Vert_{B^{{r}}_{{p},{q}}(G)}.$$ Now, we compute both sides as $|h|\rightarrow \infty.$ We observe that
\begin{align*}
\Vert T(f+\tau_hf)\Vert_{B^{\tilde{r}}_{\tilde{p},\tilde{q}}(G)} &=\left(   \sum_{l=0}^{\infty}2^{\frac{l}{\nu}\tilde{r}\tilde{q}}\Vert  \psi_{l}(I+\mathcal{R})T(f+\tau_hf)\Vert^{\tilde{q}}_{L^{\tilde{p}}(G)}      \right)^{\frac{1}{ \tilde{q}  }}\\
&=\left(   \sum_{l=0}^{\infty}2^{\frac{l}{\nu}\tilde{r}\tilde{q}}\Vert  \psi_{l}(I+\mathcal{R})Tf+\psi_{l}(I+\mathcal{R})T\tau_hf\Vert^{\tilde{q}}_{L^{\tilde{p}}(G)}      \right)^{\frac{1}{  \tilde{q}  }}.
\end{align*}
Because, $T$ and $\psi_{l}(\mathcal{R}),$ $l\in\mathbb{N}_0,$ are left-invariant, we  obtain
\begin{align*}
\lim_{|h|\rightarrow\infty}\Vert  \psi_{l}(I+\mathcal{R}) & Tf+\psi_{l}  (I+\mathcal{R})T\tau_hf\ \Vert_{L^{\tilde{p}}(G)}  \\ &=\lim_{|h|\rightarrow\infty}\Vert  \psi_{l}(I+\mathcal{R})Tf+ \tau_h\psi_{l}(I+\mathcal{R})Tf\ \Vert_{L^{\tilde{p}}(G)}\\
&=2^{\frac{1}{\tilde{p}}}\Vert  \psi_{l}(I+\mathcal{R})Tf \Vert_{L^{\tilde{p}}(G)}.
\end{align*} 
Hence $$\lim_{|h|\rightarrow \infty} \Vert T(f+\tau_hf)\Vert_{B^{\tilde{r}}_{\tilde{p},\tilde{q}}(G)} =2^{\frac{1}{\tilde{p}} }\Vert Tf\Vert_{B^{\tilde{r}}_{\tilde{p},\tilde{q}}(G)}. $$
With a similar proof we obtain $$\lim_{|h|\rightarrow \infty} \Vert f\Vert_{B^{{r}}_{{p},{q}}(G)} =2^{\frac{1}{p}}\Vert f\Vert_{B^{{r}}_{{p},{q}}(G)}. $$
Hence $$ 2^{\frac{1}{\tilde{p} } }\Vert Tf\Vert_{B^{\tilde{r}}_{\tilde{p},\tilde{q}}(G)}\leq 2^{\frac{1}{p}}\Vert T \Vert\Vert f\Vert_{B^{{r}}_{{p},{q}}(G)}.$$
The last inequality implies that $\Vert T\Vert\leq 2^{\frac{1}{p}-\frac{1}{\tilde{p}}}\Vert T\Vert.$  Thus, if $p>\tilde{p}$ then $T$ is the null operator. The proof for $\tilde{q}=\infty$ is analogous.
\end{proof}

\subsection{Fourier multipliers on Besov spaces} 

Throughout this subsection we consider (right) homogeneous and inhomogeneous Besov spaces. In order to introduce our main result of this subsection we consider the following remark on the commutativity of operators with spectral measures. 

\begin{remark}\label{calculus}
Let $R$ be a self-adjoint operator with spectral measure $E(\lambda)_{\lambda>0.}$ Then, the spectral theorem gives
$R=\int \lambda dE_{\lambda},$ and by the Stone's formula we have the following integral representation for every  spectral projection $E(\lambda),$ (see Theorem 7.17 of \cite{Weid})
\begin{equation}
E(\lambda)=\lim_{\delta\rightarrow 0^{+} }\lim_{\varepsilon\rightarrow 0^{+}}\int_{-\infty}^{\lambda+\delta}([t-i\varepsilon-R]^{-1}-[t+i\varepsilon-R]^{-1})dt.
\end{equation}
If a closed operator $T$ commutes with $R,$ then $T$  commutes with its resolvent operator $(z-R)^{-1}$ and hence with its spectral measure $(E(\lambda))_{\lambda>0}.$ Now, if $f$ is a bounded continuous function on $[0,\infty)$  and 
\begin{equation}
f(R)=\int f(\lambda)dE(\lambda),
\end{equation}
then we can write
\begin{equation}
f(R)=\lim_{\Vert P\Vert\rightarrow 0^{+} }\sum_{i=1,\lambda_i\in P}^{\infty}[f(\lambda_i)][E_{\lambda_i}-E_{\lambda_{i-1}}],
\end{equation}
where in the limit above, $P=\{0=\lambda_0<\lambda_1<\lambda_2<\cdots\}$ denotes a partition of $[0,\infty).$ So, if $T$ commutes with $R,$ then it also commutes with every bounded continuous function of $R$ defined by the functional calculus.
\end{remark}
Now we present the following theorem on Fourier multipliers in Besov spaces where we establish a connection between $L^p$ boundedness and Besov continuity of Fourier multipliers.
\begin{theorem}\label{mt2}
Let $G$ be a graded Lie group. Let $\sigma=\{\sigma(\pi):\pi\in \widehat{G}\}$ be a $\mu$-measurable field of operators in $L^{2}(\widehat{G}).$ Let us assume that the corresponding operator $T=T_{\sigma}$, given by 
\begin{equation*}
T_{\sigma}u(x)=\int_{ \widehat{G}}\mathrm{Tr}(\pi(x)\sigma(\pi)\widehat{u}(\pi))d\mu(\pi),
\end{equation*}
is a bounded operator from $L^{p_1}(G)$ into $L^{p_2}(G),$ $1\leq p_i\leq \infty.$ Then $T$ is a bounded operator from the (right) Besov space $\dot{B}^{r}_{p_1,q}(G)$ into the (right) Besov space $\dot{B}^{r}_{p_2,q}(G),$ for all $-\infty<r<\infty$ and $0<q\leq \infty.$ 
Moreover, $T$ is also a  bounded operator from the (right) Besov space ${B}^{r}_{p,q}(G)$ into the (right) Besov space ${B}^{r}_{p,q}(G)$.  
\end{theorem}
\begin{proof}
For $f\in \mathscr{S}(G)$ we have $\mathscr{F}_G(Tf)(\pi)=\sigma(\pi)\widehat{f}(\pi)=\mathscr{F}_G(f\ast (\mathscr{F}_G^{-1}\sigma))(\pi).$ If $\mathcal{R}$ is a right invariant positive Rockland operator, then for every $a\in\mathbb{C}$ (see Proposition 4.4.30 of \cite{FR2})
\begin{equation}
\mathcal{R}^{a}Tf=\mathcal{R}^{a}(f\ast\mathscr{F}_G^{-1}\sigma)=(\mathcal{R}^{a}f)\ast \mathscr{F}_G^{-1}\sigma=T(\mathcal{R}^{a}f),
\end{equation}
in particular $T$ commutes with $\mathcal{R}.$ Since $T$ commutes with $\mathcal{R},$ it commutes with its spectral measures, and with every bounded function of $\mathcal{R}$ defined by functional calculus (see, Remark \ref{calculus}). So,  
\begin{align*}
\Vert Tf \Vert_{\dot{B}^r_{p_2,q}(G)}^{q}&=\sum_{l\in\mathbb{N}_{0}}2^{\frac{lqr}{\nu}}\Vert \psi_{l}(\mathcal{R})Tf\Vert^{q}_{L^{p_2}}\\
&=\sum_{l\in\mathbb{N}_{0}}2^{\frac{lqr}{\nu}}\Vert T\psi_{l}(\mathcal{R})f\Vert^{q}_{L^{p_2}}\\
&\leq\sum_{l\in\mathbb{N}_{0}}2^{\frac{lqr}{\nu}}\Vert T\Vert^{q}_{\mathcal{L}(L^{p_1}(G), L^{p_2}(G))}\Vert\psi_{l}(\mathcal{R})f\Vert^{q}_{L^{p_1}}\\
&=\Vert T\Vert^{q}_{\mathcal{L}(L^{p_1}(G), L^{p_2}(G))}\Vert f \Vert_{\dot{B}^r_{p_1,q}(G)}^{q}.
\end{align*}
Thus $\Vert Tf \Vert_{\dot{B}^r_{p_2,q}(G)}^{q}\leq \Vert T\Vert_{\mathcal{L}(L^{p_1}(G), L^{p_2}(G))} \Vert f \Vert_{\dot{B}^{r}_{p_1,q}(G)}.$ The proof for the inhomogeneous case is similar. So we end the proof.
\end{proof}
We end this section with applications of Theorem \ref{mt2} to some examples for the Fourier multipliers bounded on $L^{p}$ and (right) Besov spaces. For notations and terminologies we follow \cite{FR2}.
\begin{example}
Let $T:\mathscr{S}(G)\rightarrow \mathscr{S}(G),$ $G$ be a graded  Lie group of homogeneous dimension $Q$. If $T$ is left-invariant and homogeneous of degree $\nu$ with
\begin{eqnarray}
-Q<\textnormal{Re}(\nu)<0,
\end{eqnarray}
and such that the right convolution kernel of $T$ is continuous away from the origin, then $T:L^{p}(G)\rightarrow L^{q}(G)$ is a bounded operator for $1<p,q<\infty$ and
\begin{eqnarray}\label{eq1}
\frac{1}{q}-\frac{1}{p}=\frac{\textnormal{Re}(\nu)}{Q}.
\end{eqnarray}
(c.f. Proposition 3.2.8 of \cite[p. 138]{FR2}). By Theorem \ref{mt2}, $T$ is a bounded operator from the right Besov space $B^{r}_{p,s}(G)$ into  the right Besov space  $B^{r}_{q,s}(G)$ with $p$ and $q$ satisfying \eqref{eq1}, $r\in\mathbb{R}$ and $0<s\leq \infty.$
\end{example}
\begin{example}
Let $T:L^{2}(G)\rightarrow L^{2}(G)$ be a bounded and left-invariant operator. Let us assume that its distributional kernel coincides on $G\setminus\{0\}$ with a continuously differentiable function $k$ with
\begin{align*}
& \int_{|x|\geq \frac{1}{2}}|k(x)|dx\leq A<\infty,\,\, \sup_{0<|x|\leq 1}|x|^{Q}|k(x)|\leq A,\\
&\sup_{0<|x|\leq 1}|x|^{Q+v_{j}}|X_{j}k(x)|\leq A, \,\,j=1,2,\ldots,
%n:=\dim(G),
\end{align*}
for some homogeneous quasi-norm $|\cdot|$ on $G$ and for some $A>0.$ Then $T$ is weak type (1,1) and bounded on $L^{p}(G),$ $1<p<\infty,$ (c.f. \cite[p. 145]{FR2}). By using Theorem \ref{mt2} we obtain the boundedness of $T$ on the  right Besov space  $B^{r}_{p,q}(G),$ $0<q\leq\infty$ and $r\in\mathbb{R}.$ 
\end{example}
\begin{example}
Let $G$ be a graded Lie group. Let $\sigma\in L^{2}(\widehat{G}).$ If 
$$\Vert \sigma \Vert_{H^{s},l.u, L,\eta, \mathcal{R} },\Vert \sigma \Vert_{H^{s},l.u, R,\eta, \mathcal{R} }<\infty$$ with $s>\frac{Q}{2},$ then the corresponding multiplier $T_{\sigma}$ extends to a bounded operator on $L^{p}(G)$ for all $1<p<\infty.$  By Theorem \ref{HMT} we have
\begin{equation}
\Vert T_{\sigma} \Vert_{ \mathcal{L}(L^p(G))    }\leq C  \max\{ \Vert \sigma \Vert_{H^{s},l.u, L,\eta, \mathcal{R} },\Vert \sigma \Vert_{H^{s},l.u, R,\eta, \mathcal{R} } \}.
\end{equation}
This is the H\"ormander-Mihlin Theorem presented in \cite{FR}. By Theorem \ref{mt2}, we obtain the boundedness of $T_{\sigma}$ on the right  Besov space $B^{r}_{p,q}(G)$ and by observing the proof of such theorem we conclude that
\begin{equation}
\Vert T_{\sigma} \Vert_{\mathcal{L}(B^{r}_{p,q}(G))}\leq C  \max\{ \Vert \sigma \Vert_{H^{s},l.u, L,\eta, \mathcal{R} },\Vert \sigma \Vert_{H^{s},l.u, R,\eta, \mathcal{R} } \}.
\end{equation}
\end{example}

\bibliographystyle{amsplain}

\end{document}